\documentclass[a4paper, 11pt]{amsart}
\usepackage{amssymb, graphicx, amsthm, amsmath, mathtools}
\usepackage{stmaryrd}
\usepackage[left=3cm, right=3cm]{geometry}
\usepackage[dvipsnames]{xcolor}
\usepackage{mathalfa}
\usepackage[full]{textcomp}
\usepackage[osf]{newtxtext}
\usepackage{mhequ}
\usepackage{mathrsfs}
\usepackage{microtype}
\usepackage{tikz-cd}
\usepackage{RhysAlphabets}
\usepackage{trees}
\usepackage{booktabs}
\usepackage{cprotect}
\usepackage{hyperref}

\DeclareSymbolFont{timesoperators}{T1}{ptm}{m}{n}
\SetSymbolFont{timesoperators}{bold}{T1}{ptm}{b}{n}
\makeatletter
\renewcommand{\operator@font}{\mathgroup\symtimesoperators}
\makeatother

\newtheorem{theorem}{\bf  {Theorem}}
\newtheorem{prop}[theorem]{\bf {Proposition}}
\newtheorem{lemma}[theorem]{\bf Lemma}
\newtheorem{definition}[theorem]{\bf Definition}
\newtheorem{corollary}[theorem]{\bf Corollary}
\newtheorem{remark}[theorem]{\bf Remark}

\numberwithin{theorem}{section}

\def\bigvert{\big| \kern-.16em \big| \kern-.16em \big|}

\def\${|\kern-.16em|\kern-.16em|}

\def\eps{\varepsilon}

\def\fs{{\mathfrak{s}}}
\def\mcD{{\mathcal{D}}}

\def\dash{\leavevmode\unskip\kern0.18em--\penalty\exhyphenpenalty\kern0.18em}
\def\slash{\leavevmode\unskip\kern0.15em/\penalty\exhyphenpenalty\kern0.15em}

\makeatletter

\DeclareRobustCommand{\TitleEquation}[2]{\texorpdfstring{\StrLeft{\f@series}{1}[\@firstchar]$\if%
		b\@firstchar\boldsymbol{#1}\else#1\fi$}{#2}}

\makeatother
\title[Global Existence for gKPZ]{Pathwise Global-in-Time Existence for the generalised KPZ Equation in the Full Subcritical Regime}

\author{Jonas Sauer$^1$ and Rhys Steele$^2$}
\address{$^1$ Friedrich-Schiller-Universit\"at, Institut f\"ur Mathematik, Inselplatz 5, 07737 Jena, Germany}
\email{
 \texttt{jonas.sauer@uni-jena.de}}
\address{$^2$ Max Planck Institute for Mathematics in the Sciences, Inselstrasse 22, 04103 Leipzig, Germany}
\email{
 \texttt{steele@mis.mpg.de}
}

\subjclass[2020]{60H17, 
60L30, 35A01, 35B51.}
\begin{document}	

\begin{abstract}
	We provide a pathwise proof of global-in-time well-posedness for the generalised KPZ equation in the full subcritical regime by an adaptation of the strategy recently applied to the generalised Parabolic Anderson Model in \cite{ES26}. Since this strategy relies crucially on the assumption that control of the supremum norm is sufficient to continue the solution, the main additional ingredient required is a treatment of the initial layer for gKPZ with merely $L^\infty$ initial data, rather than the more usual setting of $C^\theta$ initial data with $\theta > 0$. Our approach is based on an expansion around a deterministic profile followed by the introduction of an integrating factor in order to remove the terms which have critical scaling at time $0$. For pedagogical purposes, we first demonstrate the proof techniques in the case of the standard $(1+1)$-dimensional KPZ equation before turning to the more computationally involved case of generalised KPZ.
\end{abstract}

\maketitle

 \medskip

\section{Introduction}

In this article, we give a pathwise proof of global-in-time well-posedness for the BPHZ solutions\footnote{or more generally for the solutions associated to any model with suitable symmetry assumptions} of the spatially periodic generalised KPZ equation
\begin{align}\label{eq:gKPZ}
	(\partial_t - \Delta) u = \Gamma(u) (\partial_x u)^2 + g(u) \partial_x u + h(u) + \sigma(u) \xi
\end{align}
on $\mathbb{R}_+ \times \mathbb{T}$ with initial data in $C^\theta(\mathbb{T})$ for any $0< \theta$ that is valid for any driving noise $\xi$ satisfying the assumptions of any of the BPHZ Theorems \cite{CH, HS, BH23} with the additional properties that as equalities in law, one has that $\xi = - \xi$ and that $\xi(t,x) = \xi(t, -x)$ and for arbitrary smooth functions\footnote{ Here $N = N(\alpha)$ is a fixed function of the noise regularity $\alpha$ which corresponds to the usual number of derivatives that need to be controlled in the regularity structure based solution theory} $\Gamma, g, h, \sigma \in C^\infty(\mathbb{R}) \cap C_b^N(\mathbb{R})$. This includes the natural examples of fractional derivatives of space-time white noise. 

Our proof is an implementation of the same strategy that was recently used to provide a surprisingly simple proof of global-in-time existence for the generalised Parabolic Anderson Model (which corresponds to the special case $\Gamma = g = h = 0$ in gKPZ) in the recent work \cite{ES26}. Roughly speaking, the argument of \cite{ES26} has two key ingredients:
\begin{enumerate}
	\item The generalised Parabolic Anderson Model has no natural choice of origin in solution space. This means that the class of equations is invariant under constant shifts $u \mapsto u+a$. This observation allows one to obtain short-time solutions to the equation with constant initial data on a time interval that is independent of the size of the initial data with the property that the deviation of solutions from the initial data on that time interval is independent of the size of the initial data.
	\item The existing local-in-time solution theory (as provided e.g. by regularity structures \cite{H0, BCCH}) for the generalised Parabolic Anderson Model has the property that the solution map is well-defined and continuous up to the explosion time of the supremum norm. Therefore, to obtain global-in-time well-posedness it suffices to rule out explosion of the supremum norm.
\end{enumerate}
These ingredients are combined by use of a comparison principle (which can be established by arguing at fixed regularisation scale and then passing to the regularisation-free limit) which allows one to bound solutions started from arbitrary initial data between barriers that are constructed from instances of the short-time solution theory for constant initial data of the form $\pm \|u(s,\cdot)\|_\infty$.

Due to the fact that this strategy uses relatively few features of the generalised Parabolic Anderson Model, one expects that it should in fact generalise to treat a wider class of equations. In this paper, we demonstrate that this expectation is correct in the case of the generalised KPZ equation. Indeed, at the formal level, it is immediate that the generalised KPZ equation has the first property desired above since shifts in solution space can be absorbed into shifts in the functions $\Gamma, g,h, \sigma$ whilst the remaining derivative terms are independent of constant shifts. This can be made mathematically rigorous using for example the solution theory provided by regularity structures with very little effort.

The second property required above is less clear and therefore the main result of this paper can't be immediately obtained by the approach of \cite{ES26} without additional work. Indeed, to the best of our knowledge, existing solution theories for the generalised KPZ equation provide a continuous solution map which is defined up to the explosion time of the $C^\theta$ norm of the solution for any strictly positive $\theta$. The reason for this restriction can already be detected at the level of the deterministic equation. Indeed, if this equation is started from merely $L^\infty$-initial data then classical regularity theory for the heat operator would suggest that $\partial_x u$ scales at best like $t^{-1/2}$ as $t \to 0$. It then follows that $\Gamma(u) (\partial_x u)^2$ scales at best as $t^{-1}$ and therefore $u$ scales at best as $t^0$. Whilst this is sufficient to again imply that $\partial_x u$ scales as $t^{-1/2}$, there is no wiggle room. In particular, when attempting to implement a fixed-point argument, it is not clear how to extract a small constant with which to buckle the required estimates. For $C^\theta$ initial data, $(\partial_x u)^2$ instead scales as $t^{-1 + \theta}$ near time $0$ which leads to a scaling gain of $t^{\theta/2}$ in one iteration of the fixed-point map.

Therefore, the main contribution of this article that is not part of the strategy already implemented in \cite{ES26} for the generalised Parabolic Anderson Model is a proof that the explosion time of the supremum norm of the solution to gKPZ coincides with the explosion time of the stronger $C^\theta$-norm. This is achieved by providing a treatment of the equation in its initial layer with merely $L^\infty$ initial data. Our approach is based on the expansion of the solution around the deterministic profile corresponding to $g = h = \sigma = 0$ (for which a direct treatment of the $L^\infty$ initial data problem is relatively straightforward) and the introduction of an integrating factor to remove the lower order terms resulting from this expansion that remain scaling critical in the sense of their blow-up at time $0$. The main technical detail is then that one must show that these formal transformations of the equation are compatible with its renormalisation. Once these details are taken care of, we will have shown the following.

\begin{theorem}
	The spatially periodic solution of (the appropriately renormalised version of) \eqref{eq:gKPZ} associated to any symmetric\footnote{The precise definition of a symmetric model is given in Definition~\ref{def:symmetric_model} below. Roughly speaking, it corresponds to the requirement that the renormalisation is consistent with the symmetry of the equation under spatial reflections. In particular, this includes the BPHZ lift of fractional derivatives of white noise.} model exists globally-in-time. Furthermore, for any $T < \infty$, the solution map on $[0,T]$ is locally Lipschitz continuous as a function of the driving model and the initial data $\psi \in C^\theta$ for any $0 < \theta < \alpha + 2$ where $\alpha$ is the H\"older regularity of the driving noise. 
\end{theorem}

\begin{remark}
	The continuity of the solution map up to the explosion time of the $C^\theta$-norm of the solution is a special case of the general machinery of \cite{BCCH}. Therefore, the novel result is the non-explosion of this norm. We therefore assume that throughout this paper the parameter $\theta$ satisfies the assumption of this statement.
\end{remark}
Whilst to the best of our knowledge this work provides the first pathwise proof of global-in-time well-posedness for gKPZ (even as driven by white-noise rather than in the full subcritical regime), there is an expansive body of literature on long-time well-posedness for KPZ type equations. We mention by way of example that as is well-known, in the case of the standard KPZ equation, global well-posedness can be obtained by Cole-Hopf \cite{BertiniGiacomin}, by use of the invariant measure \cite{FN17} or by pathwise techniques; see e.g. the works \cite{KPZ-reloaded, ZZZ22} which contain such statements amongst other results. In the case of the white-noise driven generalised KPZ equation, we mention that the work \cite{BGN24} also obtains global well-posedness for the quasilinear variant of the equation considered here under a stronger symmetry assumption on the model than is needed here by using diffeomorphism invariance to reduce to an equation that can be treated via It\^o calculus. In comparison, an advantage of the approach implemented here is that the symmetry assumptions needed are very light and thus can be checked without heavy algebraic machinery. 

Related pathwise approaches to global well-posedness for singular SPDEs based on a priori estimates have been developed for the dynamic $\Phi^4$ model and its fractional variants \cite{MW17,MW20,CMW23,DGR23,EW24}, and in earlier work on the generalised parabolic Anderson model \cite{CFW26,SZZ26a, SZZ26b} (and variants sitting between the two \cite{EW26}). We emphasise that the methods in this paper are not expected to apply to equations with a $\Phi^4$-type polynomial nonlinearity, since such nonlinearities single out a natural origin, namely $0$, in solution space.

For expositional reasons, we choose to first present the argument in the case $\Gamma = \sigma = 1, g = h = 0$ with $\xi$ being $(1+1)$-dimensional space-time white noise which is nothing other than the KPZ equation. Whilst this is a very special case of the general result, presenting it separately allows us to isolate the core mechanism of the argument from the more involved computations that are required in the general setting. We therefore hope that by presenting the argument first in this simpler special case, we allow the reader to more easily follow the main ideas of the paper.

\subsubsection*{Acknowledgments}

The first author gratefully acknowledges the hospitality of the Max Planck Institute for Mathematics in the Sciences, at which they were a visitor when this work was carried out.

The second author gratefully acknowledges funding by the Deutsche Forschungsgemeinschaft (DFG, German Research Foundation) - CRC/TRR 388 ``Rough Analysis, Stochastic Dynamics and Related Fields'' - Project ID 516748464.
\subsubsection*{Statement on AI use}
During the preparation of this paper, the large language model ChatGPT 5.6 Sol was used as a computational assistant in the preparation of Lemma~\ref{lem:gKPZ_conj_fixed_point} and Lemma~\ref{lem:renormalised_eq}. In particular, ChatGPT 5.6 prepared the first two tables appearing in the proof of Lemma~\ref{lem:gKPZ_conj_fixed_point} using the parameters in the statement of the Lemma which were set by the authors. Beyond these two tables, the remainder of this proof is the work of the authors alone. Similarly, this model was used for a small number of mechanical computations in an early version of the proof of Lemma~\ref{lem:renormalised_eq}. ChatGPT 5.6 Sol and 6 Astra were used for literature review and typographical checks of the final version of the article.

The overall strategy, the proofs and the presentation of the paper are not AI-generated.

\section{Global-in-Time Existence for KPZ}

In this section, we illustrate our strategy of proof in the simple case of the spatially periodic KPZ equation. We note that in this case, both of the required ingredients for the strategy were in fact already available in the literature since it was established in \cite[Corollary 4.10]{KPZ-reloaded} that the relevant explosion time for the KPZ equation is that of the supremum norm by use of the Cole-Hopf transform at the level of paracontrolled solutions. Since this argument seems challenging to generalise to the setting of the generalised KPZ equation, we instead provide a different argument that does not rely on a Cole-Hopf transformation at the level of the rough equation. 

We begin with some formal calculations, performed at the level of a deterministic equation with smooth forcing, which serve as motivation for the strategy for establishing a blow-up alternative in the supremum norm. We recall that as was discussed in the introduction, the barrier is that for such initial data $(\partial_x h)^2$ will scale like $t^{-1}$ near to time $0$. Note that in this section only, we use the traditional notation of $h$ to denote the solution of KPZ. This should not be confused with the non-linearity $h$ in \eqref{eq:gKPZ} since the latter is not present in this section, which can be read independently of the rest of the paper. This scaling behaviour at time $0$ is critical and thus no small-constant can be extracted to buckle a priori estimates. To circumvent this problem, one could proceed as follows. First, given initial data $\psi$, one could expand the solution to
\begin{align*}
	(\partial_t - \Delta) h = (\partial_x h)^2 + f
\end{align*}
around the solution of 
\begin{align*}
	(\partial_t - \Delta) h^\psi = (\partial_x h^\psi)^2
\end{align*}
with the same initial data. $u = h - h^\psi$ would solve
\begin{align*}
	(\partial_t - \Delta) u = (\partial_x u)^2 + 2 \partial_x u \partial_x h^\psi + f
\end{align*}
with vanishing initial data. Unfortunately, since $\partial_x h^\psi$ has the same scaling as $\partial_x h$ at time $0$ where $h$ is the original solution, due to the second term on the right-hand side it is still not possible to close an a priori estimate for $u$ by considering this equation directly.

Instead, one could consider $e^{h^\psi} u$ which solves
\begin{align*}
	(\partial_t - \Delta)(e^{h^\psi} u) = e^{h^\psi} [(\partial_x u)^2 + f + 2 \partial_x h^\psi \partial_x u] - 2 \partial_x e^{h^\psi} \partial_x u = e^{h^\psi} [ (\partial_x u)^2 + f]
\end{align*}
where we have used the fact that $e^{h^\psi}$ is the Cole-Hopf transform of $h^\psi$ and thus solves the homogeneous heat equation. With this formulation, if $u$ scales like $t^{\delta/2}$ near to $t = 0$ then $(\partial_x u)^2$ scales like $t^{\delta - 1}$. Since $e^{h^\psi}, e^{-h^\psi}$ have bounded supremum norm (by the maximum principle), one can then expect to obtain from this rewriting that $u$ scales like $t^\delta$, which is an improvement from which one could harvest a small constant. With this heuristic in mind, we pass to the rough setting and aim to give a rigorous treatment, starting by introducing an analogue at the level of modelled distributions of the problem considered above.

Our first goal is to provide a short-time solution theory for the fixed point problem
\begin{align}\label{eq:conj_KPZ}
	W = A\,\mathcal P_\gamma\Bigl(\mathbf 1_+\,B\bigl[(DW)^2+\Xi\bigr]\Bigr),
\end{align}
where $|\Xi|_\fs = - 3/2 - \eps$ and $A, B \in \mcD^{\widetilde \gamma, 0}$ are valued in the polynomial sector for some $\widetilde \gamma > 0$. The main barrier is that applying \cite[Proposition 6.12]{H0} to the product containing $A$ yields a weight at best $0$ near the time zero hyperplane, which is insufficient to close the fixed point problem. However, this is essentially the same situation as was tackled in \cite[Appendix A]{CCHS22}. As in that paper, we write $\bar{\mathcal{D}}^{\gamma, \eta} = \mcD^{\gamma, \eta} \cap \mcD^\eta$ and $\widehat{\mcD}^{\gamma, \eta}$ for those elements $f$ of $\bar \mcD^{\gamma, \eta}$ such that $f(t,x) = 0$ for all $t \le 0$. 

We fix parameters $0 < \varepsilon < \kappa \ll 1$ and set $\alpha = |\Xi|_\fs = -\frac32 - \eps$. We also define
\begin{align*}
		\gamma=\frac32+\kappa, \qquad \eta=\frac12-\kappa, \qquad \overline\gamma=\gamma+\alpha	=\kappa-\varepsilon>0,
	\end{align*}
\begin{lemma}\label{lem:conj_fixed_point}
	Suppose that $\widetilde\gamma>\gamma$ and that $A,B\in \mcD^{\widetilde\gamma,0}$ take values in the polynomial sector.

	Then, for every $R > 0$, there exists $T_R > 0$ such that, whenever
	\begin{align*}
		\|A\|_{\widetilde\gamma,0;T_R} + \|B\|_{\widetilde\gamma,0;T_R} \leq R,
	\end{align*}
	the fixed-point problem
	\begin{align*}
		W=A\mathcal P_\gamma\mathbf 1_+B\bigl[(DW)^2+\Xi\bigr]
	\end{align*}
	on $(0, T_R] \times \mathbb{T}$ has a unique solution $W\in\widehat \mcD^{\gamma,\eta}$. Furthermore, the solution map $(A,B)\mapsto W$ is locally Lipschitz as a map $\mcD^{\widetilde\gamma,0} \times \mcD^{\widetilde\gamma,0} \to \widehat \mcD^{\gamma,\eta}$ in the sense that given a second solution $\bar{W}$ associated to $(\bar{A}, \bar{B})$
	\begin{align*}
		\|W-\overline W\|_{\gamma,\eta;T_R}	\lesssim_R	\|A-\overline A\|_{\widetilde\gamma,0;T_R} + \|B-\overline B\|_{\widetilde\gamma,0;T_R}.
	\end{align*}
	Moreover, every solution of the fixed-point equation in $\mcD^{\gamma,\eta}$ automatically belongs to $\widehat \mcD^{\gamma,\eta}$, so uniqueness also holds among solutions in $\mcD^{\gamma,\eta}$.
\end{lemma}

\begin{proof}
	We denote the norms on the compact set $(0,T] \times \mathbb{T}$ by $\|\cdot\|_{\rho,\theta;T}$. We set
	$\overline\gamma:=\gamma+\alpha=\kappa-\varepsilon>0.$
	
	We have that $DW \in \mcD_{\alpha + 1}^{\gamma-1,\eta-1}= \mcD_{-1/2 - \eps}^{\frac12+\kappa,-\frac12-\kappa}$
	where the lower index denotes the regularity of the sector in which the modelled distributions take values. 
	Therefore, by \cite[Proposition~6.12]{H0}
	$(DW)^2 \in \mcD_{2\alpha + 2}^{\overline\gamma,2\eta-2}=\mcD_{-1 - 2\eps}^{\kappa-\varepsilon,-1-2\kappa}$
	where we have used the fact that $\gamma - 1 + \alpha + 1 = \kappa - \eps = \bar \gamma$ and that $2(\eta - 1) = -1 - 2 \kappa < -1 - \kappa - \eps = (\eta - 1) + (\alpha + 1)$ since $\eps < \kappa$. Since $\widetilde\gamma>\gamma$, multiplication by $\mathbf{1}_+B$ preserves the exponents mentioned above, yielding
	$\mathbf{1}_+B(DW)^2 \in \mcD_{2 \alpha + 2}^{\overline\gamma,2\eta-2}.$
	Using the estimates in \cite[Proposition~6.12]{H0} we have therefore obtained
	\begin{align}
		\|\mathbf{1}_+B(DW)^2\|_{\overline\gamma,2\eta-2;T} &\lesssim \|B\|_{\widetilde\gamma,0;T} \|W\|_{\gamma,\eta;T}^2,
		\label{eq:quadratic-bound}
		\\
		\|\mathbf{1}_+B(DW)^2-\mathbf{1}_+B(D\overline W)^2\|_{\overline\gamma,2\eta-2;T} &\lesssim \|B\|_{\widetilde\gamma,0;T} \|W-\overline W\|_{\gamma,\eta;T}(\|W\|_{\gamma,\eta;T} + \|\overline W\|_{\gamma,\eta;T}).
		\label{eq:quadratic-difference-bound}
	\end{align} 
	Since $2\alpha + 2 = - 1 - 2\eps > -1 - 2\kappa = 2\eta - 2$, we have that $\mathbf{1}_+B(DW)^2 \in \widehat{\mcD}_{2\alpha + 2}^{\overline \gamma, 2 \eta - 2}$. 

	The constant modelled distribution $z\mapsto\Xi$ can be viewed as an element of $D_P^{\overline\gamma,\overline\gamma}$ taking values in a sector of regularity $\alpha$. Therefore $\|\mathbf{1}_+B\Xi\|_{\overline\gamma,\alpha;T}	\lesssim \|B\|_{\widetilde\gamma,0;T}$
	and $\mathbf{1}_+ B \Xi \in \widehat{\mcD}_{\alpha}^{\bar \gamma, \alpha}$. We split the integrated expression into
	\begin{align*}
		U_{\mathrm{nl}}(B, W) &:= \mathcal P_\gamma\bigl(\mathbf 1_+B(DW)^2\bigr),
		\\
		U_\Xi(B) &:= \mathcal P_\gamma\bigl(\mathbf 1_+B\Xi\bigr).
	\end{align*}
	By \cite[Theorem~7.1, Lemma~7.3]{H0},
	we therefore obtain
	\begin{align}\label{eq:nonlinear-schauder}
		&\|U_{\mathrm{nl}}(B, W)\|_{\gamma,\eta;T}	\lesssim T^{\eta/2} \|\mathbf{1}_+ B(DW)^2\|_{\overline\gamma,2\eta-2;T} \lesssim T^{\eta/2}\| B\|_{\widetilde\gamma,0;T} \|W\|_{\gamma,\eta;T}^2.
	\end{align}
	The analogous estimate holds for differences. In the case of $U_\Xi(B)$, we have that $\mathbf{1}_+B\Xi \in \widehat{\mcD}_\alpha^{\overline\gamma,\alpha}$. Proceeding similarly, we obtain
	\begin{align}\label{eq:noise-schauder}
		\|U_\Xi(B)\|_{\gamma,\eta;T} \lesssim T^{\overline\gamma/2} \|B\Xi\|_{\overline\gamma,\alpha;T}
	\end{align}
	since $(\alpha + 2)-\eta=\kappa-\varepsilon=\overline\gamma>0$.

	We furthermore have that $U_{\mathrm{nl}}(B,W),U_\Xi(B) \in \widehat{\mcD}_{0}^{\gamma, \eta}$. Indeed, writing $G = K + R$ for the usual decomposition of the heat kernel, this follows from
	 \cite[Theorem~A.9]{CCHS22} in the case of the contribution corresponding to $K$ and the fact that $R$ is smooth and non-anticipative for the contributions corresponding to $R$.
	 We now define
	\begin{align*}
		\Phi_{A,B}(W)
		:=
		A\bigl(U_{\mathrm{nl}}(B,W)+U_\Xi(B)\bigr).
	\end{align*}
	We note that since $U_{\mathrm{nl}}(B,W), U_\Xi(B)$ vanish on negative times, $\Phi_{A,B}(W) = \mathbf{1}_+ \Phi_{A,B}(W)$. We also note that $A \in \mcD_0^{\widetilde \gamma, 0}$ implies that $\mathbf{1}_+ A \in \widehat{\mcD}_0^{\widetilde \gamma, 0}$. Combining the above facts, \cite[Lemma~A.8]{CCHS22} yields
	\begin{align}
		\|\Phi_{A,B}(W)\|_{\gamma,\eta;T}
		\leq
		C_R\left(
			T^{\overline\gamma/2}
			+T^{\eta/2}\|W\|_{\gamma,\eta;T}^2
		\right)
		\label{eq:map-bound}
	\end{align}
	uniformly over $A, B$ and the model having size at most $R$ in their respective norms.
	Similarly, 
	\begin{align}
		\|\Phi_{A,B}(W)-\Phi_{A,B}(\overline W)\|_{\gamma,\eta;T}
		\leq
		C_{R} T^{\eta/2}
		\|W-\overline W\|_{\gamma,\eta;T}  (\|W\|_{\gamma,\eta;T} + \|\overline W\|_{\gamma,\eta;T})
		\label{eq:contraction-bound}
	\end{align}
	Therefore, choosing $T>0$ sufficiently small, it follows from \eqref{eq:map-bound} and \eqref{eq:contraction-bound} that $\Phi_{A,B}$ maps the ball of radius $1$ in $W$-space into itself and is a strict contraction. This gives the required
	fixed point. Uniqueness without restricting to a ball of radius $1$ follows similarly to the continuation argument in \cite[Theorem~7.8]{H0}.

	Finally, the difference versions of the preceding estimates yield, for $(A,B)$ and $(\overline A,\overline B)$ in a fixed bounded set and the corresponding arguments $W,\overline W$ in a fixed ball,
	\begin{align*}
		&\|\Phi_{A,B}(W)-
		\Phi_{\overline A,\overline B}(\overline W)\|_{\gamma,\eta;T}
		\\
		&\qquad\leq
		q\|W-\overline W\|_{\gamma,\eta;T}
		+C_{R}\left(
			\|A-\overline A\|_{\widetilde\gamma,0;T}
			+\|B-\overline B\|_{\widetilde\gamma,0;T}
		\right),
	\end{align*}
	where $q<1$ after decreasing $T$ uniformly on that bounded set. Evaluating this inequality at the two fixed points and absorbing the first term gives
	\begin{align*}
		\|W-\overline W\|_{\gamma,\eta;T}
		\lesssim_{R}
		\|A-\overline A\|_{\widetilde\gamma,0;T}
		+\|B-\overline B\|_{\widetilde\gamma,0;T}.
	\end{align*}
	This is the asserted local Lipschitz dependence.

	The final claim about the uniqueness of $W$ in the larger class $\mcD^{\gamma, \eta}$ follows from the fact that in the proof above, we did not use the assumption that $W \in \widehat{\mcD}^{\gamma, \eta}$ as an input but rather deduced it from the fixed point.
\end{proof}
The following result is an easy corollary of the deterministic Cole-Hopf transform and regularity theory for the heat equation.
Given that this proof is a simplified version of the one provided in Lemma~\ref{lem:gKPZ_PDE_control}, we omit it here for brevity.
\begin{lemma}\label{lem:psi_lift}
	Given $\psi \in L^\infty$, define $h^\psi$ to be the solution of
	\begin{align*}
		(\partial_t - \Delta) h^\psi = (\partial_x h^\psi)^2 
	\end{align*}
	with initial data $\psi$. Then $\psi \mapsto h^\psi$ is locally Lipschitz with values in $C^{\gamma, 0}$. In particular,
	\begin{align*}
		\mcE(h^\psi) &= \sum_{|k|_\fs < \widetilde{\gamma}} \frac{\partial^k e^{h^\psi}}{k!} X^k
		\\
		\mcE(-h^\psi) &= \sum_{|k|_\fs < \widetilde{\gamma}} \frac{\partial^k e^{-h^\psi}}{k!} X^k
	\end{align*}
	are well-defined elements of $D^{\widetilde{\gamma}, 0}$ that depend in a locally Lipschitz fashion on $\psi \in L^\infty$.
\end{lemma}

\begin{definition}\label{def:symmetric_model}
	We say that a renormalised model associated to a smooth noise assignment, as constructed in \cite{BHZ} is symmetric if it corresponds to an element $\ell$ of the renormalisation group with the property that $e + k_x \in 2 \mathbb{N} + 1$ implies that $\ell(\tau) = 0$ where $e$ denotes the total number of derivatives on kernel edges of $\tau$ and $k_x$ denotes the total order of the polynomial decorations in the spatial components.

	We write $\mathcal{M}^s$ for the closure of the space of symmetric admissible models associated to smooth noise assignments in the space of models.
\end{definition}
\begin{remark}
	It follows from symmetry of space-time white noise under spatial reflections that if $\varrho$ is a symmetric mollifier then the BPHZ lift of $\xi \ast \varrho^\eps$ is a symmetric model. Therefore the BPHZ lift of space-time white noise is in $\mathcal{M}^s$.
\end{remark}

\begin{lemma}\label{lem:identification}
	Suppose that $W$ is the solution of the fixed-point problem of Lemma~\ref{lem:conj_fixed_point} with respect to a model $Z = (\Pi, \Gamma) \in \mathcal{M}^s$ with $A = \mcE(-h^\psi)$ and $B = \mcE(h^\psi)$. If $Z$ corresponds to a smooth noise assignment and $w = \mathcal{R}W$, then we have that
	\begin{align}
	(\partial_t - \Delta) w = (\partial_x w)^2 + \xi_\eps + C_\eps + 2 \partial_x h^\psi \partial_x w
	\end{align}
	with $0$ initial data where $C_\eps$ is the counterterm for the KPZ equation driven by the model $Z$.
	
	In particular, it follows that for all $Z \in \mathcal{M}^s$, $h = w + h^\psi$ is the solution of KPZ with respect to $Z$. 
\end{lemma}
\begin{proof}
	We first consider the case of smooth models. We claim that
	\begin{align*}
		\mathcal{R}[\mcE(h^\psi) [(DW)^2 + \Xi]] = e^{h^\psi} [ (\partial_x w)^2 + \xi + C]
	\end{align*}
	where $C$ is the counterterm for the KPZ equation associated to the model $Z$. The fact that the counterterm associated to $Z$ is a constant as opposed to being of the form $C + \tilde{C} \partial_x h$ follows from symmetry of the model. We note that as a consequence of the fixed-point formulation, we have that
	\begin{align*}
		W = \<1> + \<d20> + 2 w_X \<d10> + \partial_x h^\psi \<X1> - \partial_x h^\psi \<1> X + 2 \<d2d10> + \mathrm{polynomial}
	\end{align*}
	where the thin edges represent an integration with no derivative, the thick edges represent an integration with derivatives and where the last term is valued in the polynomial sector and where $w_X$ is the Gubinelli derivative of $W$. This can be checked by plugging the above expression for $W$ into the fixed-point problem.

	Therefore, truncating at order $\kappa - \eps$, we have that 
	\begin{align*}
		(DW)^2 &= \<d2> + 2 \<d2d1> + 2 w_X \<d1> + 4w_X \<d1d1> + 2 \partial_x h^\psi \<Xd2> - 2 \partial_x h^\psi \<d2> X - 2 \partial_x h^\psi \<2d> + 4 \<d2d1d1> + \<d20^2> 
		\\ & \qquad + w_X^2 \mathbf{1} + 2 w_X \<d2d0>
	\end{align*}
	which in turn implies that at the same truncation order
	\begin{align*}
		\mcE(h^\psi)(DW)^2 &= e^{h^\psi}\<d2> + 2e^{h^\psi} \<d2d1> + 2e^{h^\psi} w_X \<d1> + 4e^{h^\psi}w_X \<d1d1> + 2e^{h^\psi} \partial_x h^\psi \<Xd2> - 2e^{h^\psi} \partial_x h^\psi \<d2> X 
		\\ & \qquad - 2e^{h^\psi} \partial_x h^\psi \<2d> + 4e^{h^\psi} \<d2d1d1> + e^{h^\psi} \<d20^2> 
		 + e^{h^\psi}w_X^2 \mathbf{1} + 2 e^{h^\psi} w_X \<d2d0> 
		 \\ & \qquad + e^{h^\psi} \partial_x h^\psi \<d2> X.
	\end{align*}
	This has the same form as $e^{h^\psi}$ multiplied by the usual jet for $(DH)^2$ where $H$ is the solution of KPZ at the level of modelled distributions, except for the terms in the trees $\<Xd2>, \<d2> X, \<2d>$ which appear in addition to those usually contributing to KPZ. However, each of these trees doesn't contribute to the form of the renormalisation. Indeed, the counterterm corresponding to each of these trees is $0$ by the assumption on the model. The first of them additionally admits the non-trivial contraction that contracts only $\<d2>$. However, the resulting contribution is vanishing since $\Pi_x X(x) = 0$. This completes the proof of the claim made at the start of the proof.

	We therefore have that $w$ solves
	\begin{align*}
	(\partial_t - \Delta)(e^{h^\psi} w) = e^{h^\psi} [ (\partial_x w)^2 + \xi_\eps + C_\eps] 
	\end{align*}
	which implies that
	\begin{align*}
		(\partial_t - \Delta) w &= (\partial_x w)^2 + \xi_\eps + C_\eps +  e^{h^\psi} w (\partial_t - \Delta) e^{-h^\psi} + 2 \partial_x h^\psi \partial_x (e^{h^\psi} w) e^{- h^\psi}
		\\
		&= (\partial_x w)^2 + \xi_\eps + C_\eps - 2 w (\partial_x h^\psi)^2 + 2 w (\partial_x h^\psi)^2 + 2 \partial_x h^\psi \partial_x w
	\end{align*}
	which establishes the desired result for smooth noise assignments. The final claim then follows by approximation by smooth models. 
\end{proof}
 
\begin{lemma}\label{lem:blow_up_alt}
	Suppose that $\theta \in(\kappa,\eta)$ and fix parameters $R,M>0$. Then there exists $r\in(0,\frac12]$ and a constant $C<\infty$ that depends only on the fixed parameters such that the following holds.

	Let $Z\in\mathcal M^s$ have model norm at most $M$ on
	$[-1,T+2]\times\mathbb T$, and suppose that $H^{Z,\psi}$
	solves the standard abstract fixed-point problem associated to the KPZ equation on $(0,T)\times\mathbb T$ for some $T > 0$ with respect to the model $Z$ with initial data $\psi \in C^\theta$.
	Set $h^{Z,\psi}=\mathcal R^Z H^{Z,\psi}$.

	If $0<s<T$ and
	\begin{align*}
		\|h^{Z,\psi}(s,\cdot)\|_{L^\infty}\leq R,
	\end{align*}
	then
	\begin{align*}
		\|h^{Z,\psi}(t,\cdot)\|_{C^\theta}\leq C
	\end{align*}
	for every $t<T$ satisfying $s+r\leq t\leq s+2r$.
\end{lemma}

\begin{proof}
	We first assemble the uniform estimate for the auxiliary construction. By Lemma~\ref{lem:psi_lift} we have
	\begin{align*}
		\|\mathcal E(-h^\varphi)\|_{\widetilde\gamma,0;1}
		+
		\|\mathcal E(h^\varphi)\|_{\widetilde\gamma,0;1}
		\leq C_R
	\end{align*}
	uniformly over $\|\varphi\|_{L^\infty}\leq R+1$.

	By Lemma~\ref{lem:conj_fixed_point} and its proof, we may therefore choose $r\in(0,\frac12]$ such that the auxiliary fixed-point problem with coefficients
	\begin{align*}
		A=\mathcal E(-h^\varphi),
		\qquad
		B=\mathcal E(h^\varphi)
	\end{align*}
	has a solution on $(0,2r]$ satisfying
	\begin{align*}
		\|W\|_{\gamma,\eta;2r}\leq 1,
	\end{align*}
	uniformly over these initial conditions and models of norm at most $M+1$ on $[-1,2]\times\mathbb T$.

	On $[r,2r]\times\mathbb T$, this weighted bound implies a uniform bound in $\mathcal D^\gamma$.
	Since $\theta<\alpha+2$ which is the smallest non-polynomial degree appearing in the jet for $W$, we have that
	\begin{align*}
		\sup_{u\in[r,2r]}
		\|\mathcal RW(u,\cdot)\|_{C^\theta}
		\leq C_{M,R,r}.
	\end{align*}
	By Lemma~\ref{lem:psi_lift}, we also have
	\begin{align*}
		\sup_{u\in[r,2r]}
		\|h^\varphi(u,\cdot)\|_{C^\theta}
		\leq C_R(1+r^{-\theta/2}).
	\end{align*}
	Therefore
	\begin{align}
		\sup_{u\in[r,2r]}
		\|h^\varphi(u,\cdot)+\mathcal RW(u,\cdot)\|_{C^\theta} \le C
		\label{eq:auxiliary_delayed_bound}
	\end{align}
	where $C$ is uniform over the initial conditions and models under consideration.

	We now use these auxiliary estimates to control the original solution. Fix $s$ as in the statement and choose smooth symmetric models $Z_n\to Z$.
	We let $H_n=H^{Z_n,\psi}$ be the corresponding solutions of the standard fixed-point problem associated to KPZ. We also write
	\begin{align*}
		h_n=\mathcal R^{Z_n}H_n,
		\qquad
		\varphi_n=h_n(s,\cdot).
	\end{align*}
	The standard regularity structure solution theory implies that for sufficiently large $n$, $H_n$ exists on every fixed compact time interval prior to time $T$ , and that $h_n\to h^{Z,\psi}$ locally uniformly on positive times (see \cite[Theorem~7.8, Proposition~7.11 and Corollary~7.12]{H0}). In particular,
	\begin{align*}
		\varphi_n\longrightarrow h^{Z,\psi}(s,\cdot)
		\quad\text{in }L^\infty,
	\end{align*}
	so $\|\varphi_n\|_{L^\infty}\leq R+1$ for all sufficiently large $n$.

	Let $\tau_s Z_n$ denote the model obtained by shifting the time origin to $s$. For $n$ sufficiently large, the assumed model bound and the convergence $Z_n\to Z$ give a bound of $M+1$ for these shifted models on $[-1,2]\times\mathbb T$.  We then construct $W_n$ on $(0,2r]$ as above with respect to the model $\tau_s Z_n$ and the choices $A = \mathcal E(-h^{\varphi_n})$ and $B = \mathcal E(h^{\varphi_n})$. By Lemma~\ref{lem:identification}
	\begin{align*}
		h^{\varphi_n}(u,\cdot)
		+\mathcal R^{\tau_s Z_n}W_n(u,\cdot)
	\end{align*}
	solves the same classical renormalised KPZ equation as
	$h_n(s+u,\cdot)$, with the same initial condition $\varphi_n$.
	Classical uniqueness (applied with the qualitatively smooth initial data $\varphi_n$) therefore gives
	\begin{align*}
		h_n(s+u,\cdot)
		=
		h^{\varphi_n}(u,\cdot)
		+\mathcal R^{\tau_s Z_n}W_n(u,\cdot)
	\end{align*}
	throughout their common interval of existence.

	We then fix $t<T$ such that $s+r\leq t\leq s+2r$. For all sufficiently large $n$, the original approximating
	solution exists up to time $t$, so
	\eqref{eq:auxiliary_delayed_bound} yields
	\begin{align*}
		\|h_n(t,\cdot)\|_{C^\theta}\leq C.
	\end{align*}
	Since $h_n(t,\cdot)\to h^{Z,\psi}(t,\cdot)$ uniformly and
	$0<\theta<1$, lower semicontinuity of the H\"older norm gives
	\begin{align*}
		\|h^{Z,\psi}(t,\cdot)\|_{C^\theta}
		\leq
		\liminf_{n\to\infty}\|h_n(t,\cdot)\|_{C^\theta}
		\leq C.
	\end{align*}
	The constant is independent of $s$ and $t$, which proves
	the claim.
\end{proof}
We will combine the previous result with a short-time solution theory for constant initial data that is independent of the size of that initial data.
\begin{lemma}\label{lem:constant_initial_data}
	Fix an admissible model $Z$.
	There exists $T>0$, independent of $a\in\mathbb R$, such that
	the fixed-point problem
	\begin{align*}
		H^{Z,a}
		=
		\mathcal P_\gamma^Z\mathbf1_+
		\bigl[(DH^{Z,a})^2+\Xi\bigr]
		+\mathcal G a
	\end{align*}
	has a unique solution in $\mcD^{\gamma,\eta}$ on
	$(0,T]\times\mathbb T$.
	Moreover,
	\begin{align*}
		H^{Z,a}=a\mathbf1+H^{Z,0},
		\qquad
		H^{Z,a}-a\mathbf1\in\widehat{\mcD}^{\gamma,\eta}.
	\end{align*}
	Furthermore, for every $\delta > 0$, there exists a $T_{\delta} > 0$ such that 
	\begin{align*}
		\|H^{Z,a}\|_{\gamma, \eta; T_\delta} \le {|a| + \delta.}
	\end{align*}
	This $T_\delta$ can be chosen locally uniformly over the space of models and independently of $a$.
\end{lemma}

\begin{proof}
	The standard techniques of regularity structures, or in the setting of this note an application of Lemma~\ref{lem:conj_fixed_point} with $A=B=\mathbf{1}$, yields a unique solution $H^{Z,0}\in\widehat{\mcD}^{\gamma,\eta}$ to the fixed-point problem with vanishing initial data on a sufficiently short interval $(0,T]$. Furthermore, this approach yields the existence of $T_\delta$ such that 
	\begin{align*}
	\|H^{Z,0}\|_{\gamma, \eta; T_\delta} \le \delta
	\end{align*}
	where $T_\delta$ can be chosen locally uniformly over the driving model.

	Since the heat semigroup preserves constants, its Taylor lift satisfies $\mathcal G a=a\mathbf1$ on positive times. Therefore, if $H^{Z,a} = a \mathbf{1} + H^{Z,0}$, then
	\begin{align*}
		H^{Z,a} &= \mathcal P_\gamma^Z\mathbf1_+\bigl[(DH^{Z,a})^2+\Xi\bigr] +\mathcal G a.
	\end{align*}
	This implies that $H^{Z,a}$ is a solution of the fixed-point problem for initial data $a$ on the same interval for every $a\in\mathbb R$. It is automatic that $H^{Z,a}$ satisfies the desired bound for the same value of $T_\delta$. Uniqueness for the problem started by $a$ follows similarly since given a solution $H \in D^{\gamma, \eta}$ with initial data $a$, $H - a \mathbf{1}$ solves the problem with vanishing initial data.
\end{proof}
The next ingredient is a classical comparison principle.
\begin{lemma}\label{lem:comparison}
	Suppose that for $i = 1,2$,
	\begin{align*}
		(\partial_t - \Delta) h^i = (\partial_x h^i)^2 + C + f
	\end{align*}
	with $h^i(0, \cdot) = \psi^i$ with $\psi^1 \le \psi^2$ where $C$ is a constant and $f$ is a smooth function. 
	We have that $h^1 \le h^2$.
\end{lemma}
\begin{proof}
	Since $h^1 \le h^2$ if and only if $\exp(h^1 - Ct) \le \exp(h^2 - Ct)$, it suffices to prove the same result for solutions of 
	\begin{align*}
		(\partial_t - \Delta) u^i = u^i f. 
	\end{align*}
	The result then follows by the maximum principle applied to $u^1 - u^2$. 
\end{proof}
\begin{corollary}\label{cor:comparison}
	Let $h^i$ denote the solutions to the KPZ equation with initial data $\psi^i$. If $\psi^1 \le \psi^2$ then $h^1 \le h^2$ up to the explosion time of the $C^\theta$ norms.
\end{corollary}
\begin{proof}
Apply Lemma~\ref{lem:comparison} to the KPZ equation driven by $\xi_\eps$ and then send $\eps \to 0$. The fact that the result holds up to the explosion time of the $C^\theta$ norms follows from the fact that well-posedness theory provided by regularity structures holds up to this time.
\end{proof}
\begin{theorem}
	Maximal solutions to KPZ (with respect to a symmetric model) are global in the sense that for every $T > 0$, $\sup_{0 \le s \le T} \|h(s,\cdot)\|_{C^\theta} < \infty$. 
\end{theorem}
\begin{proof}
	Suppose that $T_\mathrm{max}< \infty$ where $T_\mathrm{max}$ is the explosion time of the $C^\theta$ norm of the solution and fix a time $t_0 < T_\mathrm{max}$ and set $A_+ = \sup_{x \in \mathbb{T}} h(t_0,x)$ and $A_- = \inf_{x \in \mathbb{T}} h(t_0,x)$. We let $h_{A_\pm}$ be the solutions started at time $t_0$ with initial data $A_\pm$. By Corollary~\ref{cor:comparison}, we have that for $t \in [t_0, (t_0 + T_\delta) \wedge T_\mathrm{max})$
	\begin{align*}
		h_{A_-}(t,x) \le h(t,x) \le h_{A_+}(t,x).
	\end{align*}
	By Lemma~\ref{lem:constant_initial_data}, we then have that 
	\begin{align*}
		A_- - \delta \le h(t,x) \le A_+ + \delta.
	\end{align*}
	Therefore 
	\begin{align*}
		\sup_{t \in [t_0, (t_0 + T_\delta) \wedge T_\mathrm{max})} \|h(t, \cdot)\|_\infty \le \|h(t_0, \cdot)\|_\infty + \delta.
	\end{align*}
	Iterating this estimate yields
	\begin{align*}
		\sup_{t < T_\mathrm{max}} \|h(t, \cdot)\|_\infty \le \|h(0, \cdot)\|_\infty + \delta \Bigl \lceil \frac{T_\mathrm{max}}{T_\delta} \Bigr \rceil < \infty.
	\end{align*}
	By Lemma~\ref{lem:blow_up_alt}, we then have that
	\begin{align*}
		\sup_{T_{\mathrm{max}}/2 < t < T_{\mathrm{max}}} \|h(t, \cdot)\|_{C^{\theta}} \le C
	\end{align*}
	for some constant $C$ which depends only on $\|h(0, \cdot)\|_\infty, T_\mathrm{max}$ and $\delta$. Here we have used the fact that $r$ in Lemma~\ref{lem:blow_up_alt} can be chosen to be arbitrarily small and in particular less than $T_\mathrm{max}/4$. Since $T_\mathrm{max}$ is the explosion time of the $C^{\theta}$ norm, this is a contradiction.
\end{proof}

\section{Global-in-Time Existence for gKPZ}
We now turn to the proof of the main result of this paper. We will closely follow the overall strategy illustrated in the previous section in the case of KPZ. 

To set notation, we are interested in the (appropriately renormalised version of the) gKPZ equation
\begin{align*}
	(\partial_t - \Delta) u = \Gamma(u) (\partial_x u)^2 + g(u) \partial_x u + h(u) + \sigma(u) \xi
\end{align*}
with initial data $u(0, \cdot) = \psi$. 

The first step is to derive a fixed-point problem with $0$ initial data that will play the analogous role to \eqref{eq:conj_KPZ}.
For expositional reasons, we begin by performing the analogous heuristic calculation to the one performed in the case of the ordinary KPZ equation to deal with the critical behaviour at time $0$ for $L^\infty$ initial data.

We let $u^\psi$ denote the solution to 
\begin{align*}
	(\partial_t - \Delta) u^\psi = \Gamma(u^\psi) (\partial_x u^\psi)^2
\end{align*}
with initial data $\psi$. Then $v = u - u^\psi$ solves
\begin{align*}
	(\partial_t - \Delta) v &= \Gamma(v + u^\psi) (\partial_x v)^2 + 2 \Gamma(v + u^\psi) (\partial_x u^\psi) (\partial_x v) + [\Gamma(v + u^\psi) - \Gamma(u^\psi)] (\partial_x u^\psi)^2
	\\
	& \qquad + g(v + u^\psi) \partial_x v + g(v + u^\psi) \partial_x u^\psi + h(v + u^\psi) + \sigma(v + u^\psi) \xi
\end{align*}
In terms of scaling at time $0$, the second term on the right-hand side is the analogue of $(\partial_x h^\psi) (\partial_x u)$ in the calculation we did for ordinary KPZ. It is scaling critical. In addition, the third term on the right-hand side (which would vanish for ordinary KPZ where $\Gamma \equiv 1$) is also scaling critical since one can expect $[\Gamma(v + u^\psi) - \Gamma(u^\psi)] \sim v$ in terms of scaling at time $0$ by the Mean-Value Theorem since $v$ has vanishing initial data. 

To deal with the latter term, we add and subtract $\Gamma^\prime(u^\psi) v$ so that the Taylor expansion is to one higher order. This yields
\begin{align*}
	(\partial_t - \Delta) v &= \Gamma(v + u^\psi) (\partial_x v)^2 + 2 \Gamma(v + u^\psi) (\partial_x u^\psi) (\partial_x v) + \Gamma^\prime(u^\psi) v(\partial_x u^\psi)^2
	\\
	& \qquad  +  \sigma(v + u^\psi) \xi + \mathcal{N}_\psi(v)
\end{align*}
where
\begin{align*}
	\mathcal{N}_\psi(v) &=
	[\Gamma(v + u^\psi) - \Gamma(u^\psi) - \Gamma^\prime(u^\psi) v] (\partial_x u^\psi)^2 + g(v + u^\psi) \partial_x v 
	\\
	& \qquad + g(v + u^\psi) \partial_x u^\psi + h(v + u^\psi)
\end{align*}
groups all the terms which have subcritical scaling at time $0$. We now let $\Phi$ solve the ODE $\Phi^\prime(r) = \Gamma(r)$ with $\Phi(0) = 0$ and consider $w = e^{\Phi(u^\psi)} v$. We have
\begin{align*}
	(\partial_t - \Delta) e^{\Phi(u^\psi)} &= e^{\Phi(u^\psi)} [\Phi'(u^\psi) (\partial_t - \Delta)u^\psi - (\Phi''(u^\psi) + \Phi'(u^\psi)^2)(\partial_x u^\psi)^2]
	\\
	& = - e^{\Phi(u^\psi)} \Gamma'(u^\psi) (\partial_x u^\psi)^2.
\end{align*}
This yields
\begin{align*}
	(\partial_t - \Delta) w &= e^{\Phi(u^\psi)} (\partial_t - \Delta) v + v (\partial_t - \Delta)e^{\Phi(u^\psi)} - 2 \partial_x v \partial_x e^{\Phi(u^\psi)} 
	\\
	& = e^{\Phi(u^\psi)} \bigl [ \Gamma(v + u^\psi) (\partial_x v)^2 + 2 \Gamma(v + u^\psi) (\partial_x u^\psi) (\partial_x v) 
	\\
	& \qquad  +  \sigma(v + u^\psi) \xi + \mathcal{N}_\psi(v) - 2 \Phi'(u^\psi) \partial_x v \partial_x u^\psi \bigr ]
	\\
	& = e^{\Phi(u^\psi)} [\Gamma(v + u^\psi) (\partial_x v)^2 + 2 (\Gamma(v + u^\psi) - \Gamma(u^\psi)) \partial_x u^\psi \partial_x v 
	\\
	& \qquad + \sigma(v + u^\psi) \xi + \mathcal{N}_\psi(v)]
\end{align*}
where now all terms on the right hand side formally lead to subcritical scaling since $\Gamma(v + u^\psi) - \Gamma(u^\psi) \sim v$ near to time $0$. 

This suggests to start by looking at the fixed point problem
\begin{align*}
	\mcV & = \mcE(- \Phi(u^\psi)) \mathcal{P}_\gamma \mathbf{1}_+ \mcE(\Phi(u^\psi)) \Bigl [ \Gamma(\mcV + \mcU^\psi) (D\mcV)^2 + 2 \hat{\Gamma}^0(\mcV, \mcU^\psi) D \mcU^\psi D \mcV + \sigma(\mcV + \mcU^\psi) \Xi 
	\\
	& \qquad + \hat{\Gamma}^1(\mcV, \mcU^\psi)(D\mcU^\psi)^2 + g(\mcV + \mcU^\psi) D (\mcV + \mcU^\psi) + h(\mcV + \mcU^\psi) \Bigr ]
\end{align*}
where $\mcU^\psi$ is the lift of $u^\psi$ via its Taylor jet and where 
\begin{align*}
	\hat{\Gamma}^0(\mcV, \mcU^\psi) &= \mcV \int_0^1 \Gamma'(\mcU^\psi + \theta \mcV) d\theta = \mcV \widetilde{\Gamma}^0(\mcV, \mcU^\psi),
	\\
	\hat{\Gamma}^1(\mcV, \mcU^\psi) &= \mcV^2 \int_0^1 (1- \theta) \Gamma''(\mcU^\psi + \theta \mcV) d \theta = \mcV^2 \widetilde{\Gamma}^1(\mcV, \mcU^\psi).
\end{align*}
Note that $\widetilde{\Gamma}^i$ are smooth functions since $\Gamma$ is a smooth function. 

We now suppose that $|\Xi|_\fs = \alpha \in (-2, -3/2)$. The restriction to noise regularity at least as bad as $1+1$-dimensional space-time white noise is only used to simplify the tracking of parameters. 

In what follows, we let $\gamma = \eps - \alpha, \eta = \alpha + 2 - \eps$ and $\bar{\gamma} = \gamma + \alpha = \eps$ where $0 < \eps \le \frac{2 + \alpha}{2}$.
\begin{lemma}\label{lem:gKPZ_conj_fixed_point}
	Suppose that $\widetilde\gamma>\gamma$ and that $A,B,C \in \mcD^{\widetilde\gamma,0}$ take values in the polynomial sector.

	Then, for every $R > 0$, there exists $T_R > 0$ such that, whenever
	\begin{align*}
		\|A\|_{\widetilde\gamma,0;T_R} + \|B\|_{\widetilde\gamma,0;T_R} + \|C\|_{\widetilde \gamma, 0; T_R} \leq R,
	\end{align*}
	the fixed-point problem
	\begin{align*}
		\mcV& =A\mathcal P_\gamma\mathbf 1_+B\Bigl [  \Gamma(\mcV + C) (D\mcV)^2 + 2 \widetilde{\Gamma}^0(\mcV, C) \mcV D C D \mcV + \sigma(\mcV + C) \Xi 
	\\
	& \qquad + \widetilde{\Gamma}^1(\mcV, C) \mcV^2 (DC)^2 + g(\mcV + C) D (\mcV + C) + h(\mcV + C) \Bigr ]
	\end{align*}
	on $(0, T_R] \times \mathbb{T}$ has a unique solution $\mcV \in\widehat \mcD^{\gamma,\eta}$. Furthermore, the solution map $(A,B,C)\mapsto \mcV$ is locally Lipschitz as a map $\mcD^{\widetilde\gamma,0} \times \mcD^{\widetilde\gamma,0} \times \mcD^{\widetilde\gamma,0} \to \widehat \mcD^{\gamma,\eta}$ in the sense that given a second solution $\bar{\mcV}$ associated to $(\bar{A}, \bar{B}, \bar{C})$
	\begin{align*}
		\|\mcV-\bar \mcV\|_{\gamma,\eta;T_R}	\lesssim_R	\|A-\bar A\|_{\widetilde\gamma,0;T_R} + \|B-\bar B\|_{\widetilde\gamma,0;T_R} + \|C-\bar C\|_{\widetilde\gamma,0;T_R}.
	\end{align*}
\end{lemma}
\begin{proof}
	For a solution $\mcV \in \widehat{\mcD}_0^{\gamma, \eta}$ of the fixed point problem, we have that 
	\medskip
	
	\begin{center}
	\renewcommand{\arraystretch}{1.25}
	\begin{tabular}{@{}lccc@{}}
		\toprule
		\text{Modelled distribution}
		& $\gamma_{\star}$
		& $\eta_{\star}$
		& \text{sector regularity}
		\\
		\midrule
		$A,B,C$
		& $\widetilde\gamma$
		& $0$
		& $0$
		\\
		$\mcV$
		& $\gamma=\eps-\alpha$
		& $\eta=\alpha+2-\eps$
		& $0$
		\\
		$DC$
		& $\widetilde\gamma-1$
		& $-1$
		& $0$
		\\
		$D\mcV$
		& $\gamma-1=\eps-\alpha-1$
		& $\eta-1=\alpha+1-\eps$
		& $\alpha+1$
		\\
		$C+\mcV$
		& $\gamma$
		& $0$
		& $0$
		\\
		$\Gamma(C+\mcV)$
		& $\gamma$
		& $0$
		& $0$
		\\
		$\widetilde{\Gamma}^{0}(\mcV,C),
		\ \widetilde{\Gamma}^{1}(\mcV,C)$
		& $\gamma$
		& $0$
		& $0$
		\\
		$\sigma(C+\mcV),\ g(C+\mcV),\ h(C+\mcV)$
		& $\gamma$
		& $0$
		& $0$
		\\
		$D(C+\mcV)$
		& $\gamma-1$
		& $-1$
		& $\alpha+1$
		\\
		$\Xi$
		& $\widetilde{\gamma}$
		& $\widetilde{\gamma}$
		& $\alpha$
		\\
		\bottomrule
	\end{tabular}
\end{center}
It follows that $(D \mcV)^2 \in \mcD_{2\alpha+2}^{\gamma + \alpha, 2 \eta - 2}$. Since $\tilde{\gamma} \ge \gamma$, we may also simplify by taking $DC \in \mcD_0^{\gamma - 1, -1}$. Using only \cite[Proposition 6.12]{H0}, we would get 
\begin{center}
	\renewcommand{\arraystretch}{1.25}
	\begin{tabular}{@{}lccc@{}}
		\toprule
		Term
		& $\gamma_{\star}$
		& $\eta_{\star}$
		& sector regularity
		\\
		\midrule
		$\Gamma(C+\mcV)(D\mcV)^2$
		& $\bar\gamma=\eps$
		& $2\eta-2=2\alpha+2-2\eps$
		& $2\alpha+2$
		\\
		$\widetilde{\Gamma}^{0}(\mcV,C)\mcV\,DC\,D\mcV$
		& $\bar\gamma=\eps$
		& $\eta-2=\alpha-\eps$
		& $\alpha+1$
		\\
		$\sigma(C+\mcV)\Xi$
		& $\bar\gamma=\eps$
		& $\alpha$
		& $\alpha$
		\\
		$\widetilde{\Gamma}^{1}(\mcV,C)\mcV^2(DC)^2$
		& $\gamma-1=\eps-\alpha-1$
		& $-2$
		& $0$
		\\
		$g(C+\mcV)D(C+\mcV)$
		& $\gamma-1=\eps-\alpha-1$
		& $-1$
		& $\alpha+1$
		\\
		$h(C+\mcV)$
		& $\gamma=\eps-\alpha$
		& $0$
		& $0$
		\\
		\bottomrule
	\end{tabular}
\end{center}
	We see that the entries in the second and fourth row are already problematic in the sense that even in the case $A = B = \mathbf{1}$, they would lead to a time $0$ weight for which the fixed-point problem would not close. However, using idempotence of $\mathbf{1}_+$, we can benefit from the fact that $\mcV \in \widehat{\mcD}_0^{\gamma, \eta}, \mathbf{1}_+ D\mcV \in \widehat{\mcD}_{\alpha + 1}^{\gamma - 1, \eta - 1}$ and $\mathbf{1}_+ DC \in \widehat{\mcD}_{0}^{\gamma - 1, -1}$ where in the latter two cases, we have used the fact that $\eta_\star$ is less than or equal to the regularity of the corresponding sector. Therefore, using \cite[Lemma A.8]{CCHS22} in place of \cite[Proposition 6.12]{H0}, yields
	\begin{align*}
		\mathbf{1}_+ \mcV DC D\mcV \in \widehat{\mcD}_{\alpha + 1}^{\bar{\gamma}, 2 \eta - 2}, \qquad \mathbf{1}_+ \mcV^2 (DC)^2 \in \widehat{\mcD}_0^{\gamma - 1, 2 \eta - 2}.
	\end{align*}
	Multiplying by $\widetilde{\Gamma}^i(\mcV, C)$ using \cite[Proposition 6.12]{H0} then improves the entries of the third columns to $2\eta - 2$ in each case. 

	Since multiplication by $\mathbf{1}_+ B$ preserves each of the exponents, we obtain
	\begin{center}
	\renewcommand{\arraystretch}{1.25}
	\begin{tabular}{@{}lccc@{}}
		\toprule
		Term
		& $\gamma_{\star}$
		& $\eta_{\star}$
		& sector regularity
		\\
		\midrule
		$\mathbf{1}_+ B \Gamma(C+\mcV)(D\mcV)^2$
		& $\bar\gamma=\eps$
		& $2\eta-2=2\alpha+2-2\eps$
		& $2\alpha+2$
		\\
		$\mathbf{1}_+ B \widetilde{\Gamma}^{0}(\mcV,C)\mcV\,DC\,D\mcV$
		& $\bar\gamma=\eps$
		& $2\eta-2=2\alpha + 2 -2\eps$
		& $\alpha+1$
		\\
		$\mathbf{1}_+ B \sigma(C+\mcV)\Xi$
		& $\bar\gamma=\eps$
		& $\alpha$
		& $\alpha$
		\\
		$ \mathbf{1}_+ B \widetilde{\Gamma}^{1}(\mcV,C)\mcV^2(DC)^2$
		& $\gamma-1=\eps-\alpha-1$
		& $2\eta-2=2\alpha + 2 - 2\eps$
		& $0$
		\\
		$\mathbf{1}_+ B g(C+\mcV)D(C+\mcV)$
		& $\gamma-1=\eps-\alpha-1$
		& $-1$
		& $\alpha+1$
		\\
		$\mathbf{1}_+ B h(C+\mcV)$
		& $\gamma=\eps-\alpha$
		& $0$
		& $0$
		\\
		\bottomrule
	\end{tabular}
\end{center}

We note that in each case, $\eta_\star$ is less than or equal to the regularity of the corresponding sector. Therefore, each of these terms lives in $\widehat{\mcD}^{\gamma_\star, \eta_\star}$ for the corresponding $\gamma_\star$ and $\eta_\star$. Integrating using \cite[Theorem A.9]{CCHS22}, we get that
\begin{align*}
	\mathcal P_\gamma\mathbf 1_+B\Bigl [ & \Gamma(\mcV + C) (D\mcV)^2 + 2 \widetilde{\Gamma}^0(\mcV, C) \mcV D C D \mcV + \sigma(\mcV + C) \Xi 
	\\
	& \qquad + \widetilde{\Gamma}^1(\mcV, C) \mcV^2 (DC)^2 + g(\mcV + C) D (\mcV + C) + h(\mcV + C) \Bigr ] \in \widehat{D}_0^{2 + \eps, \alpha + 2}.
\end{align*}
Having performed the power-counting, we now more carefully establish the norm estimates. Let $\mathcal N_{B,C}(\mcV)$ denote the expression obtained by multiplying
the terms inside the square brackets by $\mathbf 1_+B$. By the above we have that
\begin{align*}
	\mathcal N_{B,C}(\mcV)
	\in\widehat{\mcD}_{\alpha}^{\bar\gamma,\alpha}.
\end{align*}
Moreover, the corresponding multiplication and composition estimates imply
that, for every $M>0$, uniformly over $T\leq1$ and
$\|\mcV\|_{\gamma,\eta;T}\leq M$,
\begin{align*}
	\|\mathcal N_{B,C}(\mcV)\|_{\bar\gamma,\alpha;T}
	&\lesssim_{R,M}1,                             
	\\
	\|\mathcal N_{B,C}(\mcV)
		-\mathcal N_{\bar B,\bar C}(\bar{\mcV})\|_{\bar\gamma,\alpha;T}
	&\lesssim_{R,M}
	\Bigl(
		\|\mcV-\bar{\mcV}\|_{\gamma,\eta;T}
		+\|B-\bar B\|_{\widetilde\gamma,0;T}
		+\|C-\bar C\|_{\widetilde\gamma,0;T}
	\Bigr).                                                       
\end{align*}
Here and below, the implicit constants are allowed to depend on the fixed
model and on the nonlinearities.

Consequently,
\begin{align*}
	\|\mathcal P_\gamma\mathcal N_{B,C}(\mcV)\|_{\gamma,\eta;T}
	&\lesssim
	T^{\eps/2}
	\|\mathcal N_{B,C}(\mcV)\|_{\bar\gamma,\alpha;T},                \\
	\|\mathcal P_\gamma\mathcal N_{B,C}(\mcV)
		-\mathcal P_\gamma\mathcal N_{\bar B,\bar C}(\bar{\mcV})
		\|_{\gamma,\eta;T}
	&\lesssim_{R,M}
	T^{\eps/2}
	\Bigl(
		\|\mcV-\bar{\mcV}\|_{\gamma,\eta;T}
		+\|B-\bar B\|_{\widetilde\gamma,0;T}
		\\ &
		\qquad \qquad \qquad +\|C-\bar C\|_{\widetilde\gamma,0;T}
	\Bigr).                    
\end{align*}

Define
\begin{align*}
	\Phi_{A,B,C}(\mcV)
	=A\mathcal P_\gamma\mathcal N_{B,C}(\mcV).
\end{align*}
Multiplication by $\mathbf 1_+A$ and the preceding estimates give
\begin{align*}
	\|\Phi_{A,B,C}(\mcV)\|_{\gamma,\eta;T}
	&\lesssim_{R,M}T^{\eps/2},\\
	\|\Phi_{A,B,C}(\mcV)
		-\Phi_{\bar A,\bar B,\bar C}(\bar{\mcV})\|_{\gamma,\eta;T}
	&\lesssim_{R,M}T^{\eps/2}
	\Bigl(
		\|\mcV-\bar{\mcV}\|_{\gamma,\eta;T}
		+\|A-\bar A\|_{\widetilde\gamma,0;T}	
	\\ & \qquad \qquad \qquad 
	+\|B-\bar B\|_{\widetilde\gamma,0;T}
		+\|C-\bar C\|_{\widetilde\gamma,0;T}
	\Bigr).
\end{align*}
Taking $M=1$ and then choosing $T_R$ sufficiently small shows that
$\Phi_{A,B,C}$ maps the unit ball into itself and is a contraction so that the first part of the statement follows by the usual fixed-point approach. The
last estimate, followed by absorption of the term involving
$\|\mcV-\bar{\mcV}\|_{\gamma,\eta;T_R}$, gives the asserted local
Lipschitz estimate. Uniqueness without restricting to a ball of radius $1$ again follows similarly to the continuation argument in \cite[Theorem~7.8]{H0}.
\end{proof}

We next turn to controlling the deterministic profile.

\begin{lemma}\label{lem:gKPZ_PDE_control}
	Given $\psi \in L^\infty$, define $u^\psi$ to be the solution of
	\begin{align*}
		(\partial_t - \Delta) u^\psi = \Gamma(u^\psi) (\partial_x u^\psi)^2 
	\end{align*}
	with initial data $\psi$. Then $\psi \mapsto u^\psi$ is locally Lipschitz with values in $C^{\gamma, 0}$. In particular,
	\begin{align*}
		\mcE(\Phi(u^\psi)) &= \sum_{|k|_\fs < \widetilde{\gamma}} \frac{\partial^k e^{\Phi(u^\psi)}}{k!} X^k
		\\
		\mcE(-\Phi(u^\psi)) &= \sum_{|k|_\fs < \widetilde{\gamma}} \frac{\partial^k e^{-\Phi(u^\psi)}}{k!} X^k
		\\
		\mcU^\psi &= \sum_{|k|_\fs < \widetilde{\gamma}} \frac{\partial^k {u^\psi}}{k!} X^k
	\end{align*}
	are well-defined elements of $D^{\widetilde{\gamma}, 0}$ that depend in a locally Lipschitz fashion on $\psi \in L^\infty$.
\end{lemma}
\begin{proof}
	Let $J$ solve $J'' = J' \Gamma$ with $J(0) = 0$ and $J'(0) = 1$. Explicitly, we have that
	\begin{align}\label{eq:a_ODE_solution}
		J(r) = \int_0^r \exp \Bigl (\int_0^s \Gamma(q) dq \Bigr ) ds = \int_0^r \exp(\Phi(s)) ds
	\end{align}
	In particular, this yields that $J'(r) > 0$ for all $r \in \mathbb{R}$. Therefore $J$ is strictly monotone and thus by the Inverse Function Theorem possesses a smooth inverse $Q:=J^{-1}$ on its image.
	The point of introducing this choice of $J$ is that
	\begin{align*}
		(\partial_t - \Delta) J(u^\psi) = J'(u^\psi)(\partial_t - \Delta)u^\psi - J''(u^\psi) (\partial_x u^\psi)^2 = 0
	\end{align*}
	with initial data $J(\psi)$. 

	Therefore, by standard estimates for the heat equation, we get that 
	$$\|\partial^k J(u^\psi) \|_\infty \lesssim t^{-|k|_\fs/2} \|J(\psi)\|_\infty \le t^{-|k|_\fs/2} \|\psi\|_\infty \|J'\|_{\infty, [-R,R]}.$$
	In particular, for every $\gamma > 0$, we get that
	\begin{align*}
		\|J(u^\psi)\|_{C^{\gamma, 0}} \lesssim_R \|\psi\|_\infty.
	\end{align*}
	Since $Q(0) = 0$, by \cite[Proposition 6.13]{H0} and the standard isomorphism between $C^{\gamma,0}$ and $D^{\gamma,0}$, we then have that there exists a polynomial $p$ such that
	\begin{align}\label{eq:a_ODE_solution_est}
		\|u^\psi\|_{C^\gamma, 0} \lesssim_R p(\|J(\psi)\|_\infty) \|J(u^\psi)\|_{C^\gamma, 0} \lesssim_R \tilde{p}(\|\psi\|_\infty)
	\end{align}
	for some other polynomial $\tilde{p}$. Note that here we have used the fact that if the derivative of $Q$ are bounded on $[J(-R), J(R)]$ with $R = \|\psi\|_\infty$.
	
	For the local Lipschitz dependency, we proceed analogously. That is, we note that for two initial data $\psi$ and $\phi$, the function $w:=J(u^\psi)-J(u^\phi)$ solves the homogeneous heat equation with initial datum $J(\psi)-J(\phi)$.
	By a similar argument as for \eqref{eq:a_ODE_solution_est}, $w$ is controlled by $J(\psi)-J(\phi)$.
	Since $J$ is smooth, this updates to control of $w$ in terms of $\psi-\phi$, and since $Q=J^{-1}$ is smooth, the same control holds for $u^\psi-u^\phi=Q(J(u^\psi))-Q(J(u^\phi))$.
\end{proof}
\begin{remark}
Since the range of $J$ is generally not the entire real line, the inverse $Q$ is defined only on an interval.
This restriction is not problematic, since we only ever need to make sense of $Q$ on the range of $J$.
In fact, we will only apply $Q$ to functions and modelled distributions that arise as a transform under $J$.
More precisely, the only modelled distributions we will apply $Q$ to are of the form $J(\mcU^\psi)$, where $\mcU^\psi$ is the lift of $u^\psi$.
Therefore, the composition formula in Theorem 4.16 of \cite{H0} implies that we only need to know $Q$ on the image of $J$.
\end{remark}

\begin{lemma}\label{lem:gKPZ_identification}
	Suppose that $\mcV$ is the solution of the fixed-point problem of Lemma~\ref{lem:gKPZ_conj_fixed_point} with respect to a model $Z = (\Pi, \Gamma) \in \mathcal{M}^s$ with $A = \mcE(-\Phi(u^\psi))$, $B = \mcE(\Phi(u^\psi))$ and $C = \mcU^\psi$. If $Z$ corresponds to a smooth noise assignment and $v = \mathcal{R}\mcV$, then we have that
	\begin{align*}
	(\partial_t - \Delta) v &= \Gamma(v + u^\psi) (\partial_x v)^2 + 2 \Gamma(v + u^\psi) \partial_x u^\psi \partial_x v + \Gamma'(u^\psi) v (\partial_x u^\psi)^2
	\\
	& \qquad + \sigma(v + u^\psi) \xi + \mathcal{N}_\psi(v) 
	 + \mathfrak{C}(v+ u^\psi)
	\end{align*}
	with $0$ initial data where $\mathfrak{C}(v+u^\psi)$ is the usual counterterm for the gKPZ equation driven by the model $Z$.
	
	In particular, it follows that for all $Z \in \mathcal{M}^s$, $u = v + u^\psi$ is the solution of gKPZ with respect to $Z$. 
\end{lemma}
\begin{proof}
	As in Lemma~\ref{lem:identification}, it suffices to prove the claim regarding smooth models. Since the proof is computationally involved, we defer the details to Section~\ref{sec:gKPZ_identification}. In particular, this result follows from the combination of Lemma~\ref{lem:renormalised_eq} and Lemma~\ref{lem:Upsilon}.
\end{proof}

\begin{lemma}\label{lem:gKPZ_blow_up_alt}
	Suppose that $\theta\in(0,\eta)$ and fix parameters $R,M>0$. Then there exists $r\in(0,\frac12]$ and a constant $C<\infty$ that depends only on the fixed parameters such that the following holds.

	Let $Z\in\mathcal M^s$ have model norm at most $M$ on
	$[-1,T+2]\times\mathbb T$, and suppose that $\mcU^{Z,\psi}$
	solves the standard abstract fixed-point problem associated to the gKPZ equation on $(0,T)\times\mathbb T$ for some $T > 0$ with respect to the model $Z$ with initial data $\psi \in C^\theta$.
	Set $u^{Z,\psi}=\mathcal R^Z \mcU^{Z,\psi}$.

	If $0<s<T$ and
	\begin{align*}
		\|u^{Z,\psi}(s,\cdot)\|_{L^\infty}\leq R,
	\end{align*}
	then
	\begin{align*}
		\|u^{Z,\psi}(t,\cdot)\|_{C^\theta}\leq C
	\end{align*}
	for every $t<T$ satisfying $s+r\leq t\leq s+2r$.
\end{lemma}
\begin{proof}
	By Lemma~\ref{lem:gKPZ_conj_fixed_point} and its proof along with Lemma~\ref{lem:gKPZ_PDE_control}, we may choose $r\in(0,\frac12]$ such that the auxiliary fixed-point problem with coefficients
	\begin{align*}
		A=\mathcal E(-\Phi(u^\varphi)),
		\qquad
		B=\mathcal E(\Phi(u^\varphi)),
		\qquad
		C=\mcU^\varphi
	\end{align*}
	has a solution on $(0,2r]$ satisfying
	\begin{align*}
		\|\mcV\|_{\gamma,\eta;2r}\leq 1,
	\end{align*}
	uniformly over $\|\varphi\|_\infty \le R + 1$ and models of norm at most $M+1$ on $[-1,2]\times\mathbb T$.

	On $[r,2r]\times\mathbb T$, this weighted bound implies a uniform bound in $\mathcal D^\gamma$.
	Since $\theta<\alpha+2$ which is the smallest non-polynomial degree appearing in the jet for $\mcV$, we have that
	\begin{align*}
		\sup_{s\in[r,2r]}
		\|\mathcal R\mcV(s,\cdot)\|_{C^\theta}
		\leq C_{M,r}.
	\end{align*}
	By Lemma~\ref{lem:gKPZ_PDE_control}, we also have that
	\begin{align*}
		\sup_{s\in[r,2r]}
		\|u^\varphi(s,\cdot)\|_{C^\theta}
		\leq C_R(1+r^{-\theta/2}).
	\end{align*}
	Therefore
	\begin{align}
		\sup_{s\in[r,2r]}
		\|u^\varphi(s,\cdot)+\mathcal R\mcV(s,\cdot)\|_{C^\theta}
		\leq C,
		\label{eq:gKPZ_auxiliary_delayed_bound}
	\end{align}
	where $C$ is uniform over the initial conditions and models under consideration.

	We now use these auxiliary estimates to control the original solution. Fix $s$ as in the statement and choose smooth symmetric models $Z_n\to Z$.
	We let $\bar \mcU_n= \bar \mcU^{Z_n,\psi}$ be the corresponding solutions of the standard fixed-point problem associated to gKPZ. We also write
	\begin{align*}
		u_n=\mathcal R^{Z_n}\bar \mcU_n,
		\qquad
		\varphi_n=u_n(s,\cdot).
	\end{align*}
	The standard regularity structure solution theory implies that for sufficiently large $n$, $\bar \mcU_n$ exists on every fixed compact time interval prior to time $T$ , and that $u_n\to u^{Z,\psi}$ locally uniformly on positive times (see \cite[Theorem~7.8, Proposition~7.11 and Corollary~7.12]{H0}). In particular,
	\begin{align*}
		\varphi_n\longrightarrow u^{Z,\psi}(s,\cdot)
		\quad\text{in }L^\infty,
	\end{align*}
	so $\|\varphi_n\|_{L^\infty}\leq R+1$ for all sufficiently large $n$.

	Let $\tau_s Z_n$ denote the model obtained by shifting the time origin to $s$. For $n$ sufficiently large, the assumed model bound and the convergence $Z_n\to Z$ give a bound of $M+1$ for these shifted models on $[-1,2]\times\mathbb T$.  We then construct $\mcV_n$ on $(0,2r]$ as above with respect to the model $\tau_s Z_n$ and the choices $A = \mathcal E(-\Phi(u^{\varphi_n}))$, $B = \mathcal E(\Phi(u^{\varphi_n}))$ and $C = \mcU^{\varphi_n}$. By Lemma~\ref{lem:gKPZ_identification}
	\begin{align*}
		u^{\varphi_n}(\bar{s},\cdot)
		+\mathcal R^{\tau_s Z_n}\mcV_n(\bar{s},\cdot)
	\end{align*}
	solves the same classical renormalised gKPZ equation as
	$u_n(s+\bar s,\cdot)$, with the same initial condition $\varphi_n$.
	Classical uniqueness therefore gives
	\begin{align*}
		u_n(s+\bar s,\cdot)
		=
		u^{\varphi_n}(\bar s,\cdot)
		+\mathcal R^{\tau_s Z_n}\mcV_n(\bar s,\cdot)
	\end{align*}
	throughout their common interval of existence.

	We then fix $t<T$ such that $s+r\leq t\leq s+2r$. For all sufficiently large $n$, the original approximating
	solution exists up to time $t$, so
	\eqref{eq:gKPZ_auxiliary_delayed_bound} yields
	\begin{align*}
		\|u_n(t,\cdot)\|_{C^\theta}\leq C.
	\end{align*}
	Since $u_n(t,\cdot)\to u^{Z,\psi}(t,\cdot)$ uniformly and
	$0<\theta<1$, lower semicontinuity of the H\"older norm gives
	\begin{align*}
		\|u^{Z,\psi}(t,\cdot)\|_{C^\theta}
		\leq
		\liminf_{n\to\infty}\|u_n(t,\cdot)\|_{C^\theta}
		\leq C.
	\end{align*}
	The constant is independent of $s$ and $t$, which proves
	the claim.
\end{proof}

\begin{lemma}\label{lem:gKPZ_constant_initial_data}
	Fix an admissible model $Z$.
	There exists $T>0$, independent of $a\in\mathbb R$, such that
	the fixed-point problem
	\begin{align}\label{eq:gKPZ_orig_fixed_point}
		\mcU^{Z,a}
		=
		\mathcal P_\gamma^Z\mathbf1_+
		\bigl[ \Gamma(\mcU^{Z,a}) (D \mcU^{Z,a})^2 + g(\mcU^{Z,a}) D \mcU^{Z,a} + h(\mcU^{Z,a}) + \sigma(\mcU^{Z,a}) \Xi \bigr]
		+\mathcal G a
	\end{align}
	has a unique solution in $\mcD^{\gamma,\eta}$ on
	$(0,T]\times\mathbb T$.
	Moreover,
	\begin{align*}
		\mcU^{Z,a}=a\mathbf1+\widehat{\mcU}^{Z,a},
	\end{align*}
	where $\widehat{\mcU}^{Z,a}$ solves
	\begin{align*}
		\widehat \mcU^{Z,a}
		=
		\mathcal P_\gamma^Z\mathbf1_+
		\bigl[ \Gamma_a(\widehat \mcU^{Z,a}) (D \widehat \mcU^{Z,a})^2 + g_a(\widehat \mcU^{Z,a}) D \widehat \mcU^{Z,a} + h_a(\widehat \mcU^{Z,a}) + \sigma_a(\widehat \mcU^{Z,a}) \Xi \bigr]
	\end{align*}
	where given a function $f: \mathbb{R} \to \mathbb{R}$, we have set $f_a(x) = f(x + a)$. 
	In particular, for every $\delta > 0$, there exists a $T_{\delta} > 0$ such that 
	\begin{align*}
		\|\mcU^{Z,a}\|_{\gamma, \eta; T_\delta} \le {|a| + \delta}.
	\end{align*}
	This $T_\delta$ can be chosen locally uniformly over the space of models and independently of $a$.
\end{lemma}
\begin{proof}
The standard solution theory for gKPZ via regularity structures, or in the setting of this note an application of Lemma~\ref{lem:gKPZ_conj_fixed_point} with $A=B=\mathbf{1}$, $C = 0$ and each function $f \in \{\Gamma, g, h, \sigma\}$ replaced with $f_a$, yields a unique solution $\widehat{\mcU}^{Z,a}\in\widehat{\mcD}^{\gamma,\eta}$ to the fixed-point problem with vanishing initial data on a sufficiently short interval $(0,T]$. Furthermore, this approach yields the existence of $T_\delta$ such that 
	\begin{align*}
	\|\widehat{\mcU}^{Z,a}\|_{\gamma, \eta; T_\delta} \le \delta
	\end{align*}
	where $T_\delta$ can be chosen locally uniformly over the driving model and over $a$. The fact that $T_\delta$ can be chosen uniformly in $a$ follows from the fact that the constants in the solution theory exhibited in e.g. Lemma~\ref{lem:gKPZ_conj_fixed_point} depend only on $C^k$-norms of the functions $f$ which are invariant under replacing $f$ with $f_a$.

	Since the heat semigroup preserves constants, its Taylor lift satisfies $\mathcal G a=a\mathbf1$ on positive times. Therefore, if $\mcU^{Z,a} = a \mathbf{1} + \widehat \mcU^{Z,a}$, then $\mcU^{Z,a}$ solves the fixed-point problem \eqref{eq:gKPZ_orig_fixed_point}.
	This implies that $\mcU^{Z,a}$ is a solution of the fixed-point problem for initial data $a$ on the same interval for every $a\in\mathbb R$. This implies the claimed bound for $\mcU^{Z,a}$. Uniqueness for the problem started at $a$ follows similarly since given a solution $\mcU \in D^{\gamma, \eta}$ with initial data $a$, $\mcU - a \mathbf{1}$ solves the problem with vanishing initial data and nonlinearities $f_a$.
\end{proof}

We next turn to establishing the appropriate comparison principle for the smooth renormalised PDEs. We recall that the results of \cite{BCCH} imply that if $Z$ is a smooth symmetric model then the regularity structures solution associated to \eqref{eq:gKPZ} via $Z$ solves a PDE of the form
\begin{align*}
	(\partial_t - \Delta) u &= \Gamma(u) (\partial_x u)^2 + g(u) \partial_x u 
	\\
	& \qquad + h(u) + \sigma(u) \xi + \sum_{(a,b)} C^{a,b} \prod_{k \ge 0} (\Gamma^{(k)}(u))^{a_k} \prod_{j \ge 0} (\sigma^{(j)}(u))^{b_j}.
\end{align*}
This standard fact will also be recovered as part of the proofs of Lemma~\ref{lem:symm_reduction} and Lemma~\ref{lem:Upsilon} in Section~\ref{sec:gKPZ_identification} below.
\begin{prop}\label{prop:gKPZ_comparison}
	Let $\xi$ be smooth and let $\mathcal R$ be a smooth local function.
	Suppose that $u^1$ and $u^2$ solve
	\begin{align*}
		(\partial_t-\Delta)u^i
		={}&
		\Gamma(u^i)(\partial_xu^i)^2
		+g(u^i)\partial_xu^i
		+h(u^i)
		+\sigma(u^i)\xi
		+\mathcal R(u^i),
	\end{align*}
	with initial conditions $u_0^1,u_0^2\in C^\theta(\mathbb T)$ for some
	$\theta>0$. If $u_0^1\leq u_0^2$, then
	\[
		u^1(t,x)\leq u^2(t,x)
	\]
	throughout their common interval of existence.
\end{prop}

\begin{proof}
	Fix $\theta'\in(0,\theta)$, and let $\rho_\delta$ be a smooth,
	non-negative mollifier on $\mathbb T$. Set
	$
		u_0^{i,\delta}=\rho_\delta*u_0^i.
	$
	Then
	$
		u_0^{2,\delta}-u_0^{1,\delta}
		=
		\rho_\delta*(u_0^2-u_0^1)
		\geq0,
	$
	and
	$
		u_0^{i,\delta}\to u_0^i
	$
	in $C^{\theta'}(\mathbb{T})$.

	Let $u^{i,\delta}$ denote the solution of the same equation with
	initial condition $u_0^{i,\delta}$. 
		Set
	$$
		F(t,x,r,p)
		=
		\Gamma(r)p^2+g(r)p+h(r)
		+\sigma(r)\xi(t,x)+\mathcal R(r).
	$$
	For $s\in[0,1]$, write
	$$
		u_s=u^{2,\delta}
		+s(u^{1,\delta}-u^{2,\delta}),
		\qquad
		p_s=\partial_xu^{2,\delta}
		+s(\partial_xu^{1,\delta}
		-\partial_xu^{2,\delta}).
	$$
	Then $w=u^{1,\delta}-u^{2,\delta}$ solves
	$$
		(\partial_t-\Delta)w
		=b\,\partial_xw+c\,w,
	$$
	where
	$$
		b(t,x)=\int_0^1\partial_pF(t,x,u_s,p_s)\,ds,
		\qquad
		c(t,x)=\int_0^1\partial_rF(t,x,u_s,p_s)\,ds.
	$$
	On every compact interval of classical existence, $b$ and $c$ are
	bounded. Choose $K>\|c\|_{L^\infty}$ and set $z=e^{-Kt}w$. Then
	$$
		\partial_tz-\Delta z-b\partial_xz+(K-c)z=0,
		\qquad K-c>0.
	$$
	Since $z(0,\cdot)\leq0$, the weak parabolic maximum principle implies
	that $z\leq0$, and hence
	$$
		u^{1,\delta}\leq u^{2,\delta}
	$$
	throughout their common interval of existence.

	It remains to send $\delta \to 0$. Fix $T$ strictly below the common existence time of $u^1$ and $u^2$.
	Continuous dependence on the initial condition in $C^{\theta'}$
	implies that, for all sufficiently small $\delta$, the solutions
	$u^{i,\delta}$ exist up to time $T$ and
	$$
		u^{i,\delta}\longrightarrow u^i
		\qquad\text{uniformly on }[0,T]\times\mathbb T.
	$$
	Passing to the limit preserves the inequality. Since $T$ was
	arbitrary, the result follows.
\end{proof}
\begin{corollary}\label{cor:gKPZ_comparison}
	Let $u^i$ denote the solutions to the gKPZ equation with initial data $\psi^i$. If $\psi^1 \le \psi^2$ then $u^1 \le u^2$ up to the explosion times of their $C^\theta$ norms.
\end{corollary}
\begin{proof}
Apply Proposition~\ref{prop:gKPZ_comparison} to the gKPZ equation driven by $\xi_\eps$ and then send $\eps \to 0$. 
\end{proof}
\begin{theorem}
	Maximal solutions to gKPZ associated to symmetric models are global in the sense that for every $T < \infty$, 
	\begin{align*}
		\sup_{0 \le s \le T} \|u(s, \cdot)\|_{C^\theta} < \infty.
	\end{align*}
\end{theorem}
\begin{proof}
	Suppose that $T_\mathrm{max}< \infty$ where $T_\mathrm{max}$ is the explosion time of the $C^\theta$ norm of $u$. Fix a time $t_0 < T_\mathrm{max}$ and set $A_+ = \sup_{x \in \mathbb{T}} u(t_0,x)$ and $A_- = \inf_{x \in \mathbb{T}} u(t_0,x)$. We let $u_{A_\pm}$ be the solutions started at time $t_0$ with initial data $A_\pm$. By Corollary~\ref{cor:gKPZ_comparison}, we have that for $t \in [t_0, (t_0 + T_\delta) \wedge T_\mathrm{max})$
	\begin{align*}
		u_{A_-}(t,x) \le u(t,x) \le u_{A_+}(t,x).
	\end{align*}
	By Lemma~\ref{lem:gKPZ_constant_initial_data}, we then have that 
	\begin{align*}
		A_- - \delta \le u(t,x) \le A_+ + \delta.
	\end{align*}
	Therefore 
	\begin{align*}
		\sup_{t \in [t_0, (t_0 + T_\delta) \wedge T_\mathrm{max})} \|u(t, \cdot)\|_\infty \le \|u(t_0, \cdot)\|_\infty + \delta.
	\end{align*}
	Iterating this estimate yields
	\begin{align*}
		\sup_{t < T_\mathrm{max}} \|u(t, \cdot)\|_\infty \le \|u(0, \cdot)\|_\infty + \delta \Bigl \lceil \frac{T_\mathrm{max}}{T_\delta} \Bigr \rceil < \infty.
	\end{align*}
	By Lemma~\ref{lem:gKPZ_blow_up_alt}, we then have that
	\begin{align*}
		\sup_{T_{\mathrm{max}}/2 < t < T_{\mathrm{max}}} \|u(t, \cdot)\|_{C^{\theta}} \le C
	\end{align*}
	for some constant $C$ which depends only on $\|u(0, \cdot)\|_\infty, T_\mathrm{max}$ and $\delta$. We used here the fact that $r$ in Lemma~\ref{lem:gKPZ_blow_up_alt} can be chosen to be arbitrarily small and in particular less than $T_\mathrm{max}/4$. Since $T_\mathrm{max}$ is the explosion time of the $C^{\theta}$ norm, this is a contradiction.
\end{proof}

\section{Proof of Lemma~\ref{lem:gKPZ_identification}}\label{sec:gKPZ_identification}
To complete the proof, it remains to prove Lemma~\ref{lem:gKPZ_identification}; meaning that it remains to show that for a smooth, symmetric model $Z$, $u = u^\psi + \mathcal{R}^Z\mcV$ solves the gKPZ equation with counterterm associated to the model $Z$. The main difficulty is that the system for $\mcV$ is not in a form where the results of \cite{BCCH} are applicable so it's not immediate to deduce the form of the renormalised PDE. This motivates trying to rewrite the fixed point in a different form. Our first step is trying to rewrite the input modelled distributions $A, B, C$ in terms of a single modelled distribution $\mcY$. 

We let $y$ be the solution of the heat equation 
\begin{align}
	(\partial_t - \Delta)y = 0, \qquad y(0, \cdot) = J(\psi)
\end{align}

Then $u^\psi = Q(y)$ so that by injectivity of the reconstruction operator on the polynomial sector, we have that $\mcU^\psi = Q(\mcY)$ where $\mcY$ is the Taylor lift of $y$. We also note that 
$$\mcR \mcE(- \Phi(u^\psi)) = e^{- \Phi(u^\psi)} = \frac{1}{J'(u^\psi)}, \qquad \mcR \mcE(\Phi(u^\psi)) = e^{\Phi(u^\psi)} = J'(u^\psi)$$
so that
$\mcE(-\Phi(u^\psi)) = a(\mcY)$ and $\mcE(\Phi(u^\psi)) = b(\mcY)$ where 
$$a = \frac{1}{J' \circ Q}, \qquad b = J' \circ Q.$$

We set $\mcW = b(\mcY) \mcV$. We now derive the system solved by $\mcW$. The advantage of this system is that it will be in the form required to apply the results of \cite{BCCH} (see also \cite{BB26}) in order to find the form of the renormalised equation. In what follows, for readability, we will suppress the appearance of $\mcY$ in the functions so that, for example, $a$ will be a shorthand for $a(\mcY)$. We have that the system for $\mcV$ is equivalent to the system
\begin{align}
	\mcY &= \mcG J(\psi) \label{eq:W_sys}
	\\ \nonumber
	\mcW &= \mcP_\gamma \mathbf{1}_+ b \Bigl [ \Gamma(a\mcW + Q) [a D \mcW + a'\mcW D\mcY]^2 + 2 \tilde{\Gamma}^0 (a \mcW, Q)
	a Q' \mcW  D \mcY [a D \mcW + a' \mcW D \mcY] 
	\\ \nonumber
	& \qquad + \tilde{\Gamma}^1(a \mcW, Q) a^2 \mcW^2 (D \mcY)^2 (Q')^2 + g(a \mcW + Q) [ a D\mcW + a' \mcW D \mcY 
	\\ \nonumber
	& \qquad + Q' D \mcY] + h(a \mcW + Q) + \sigma(a \mcW + Q) \Xi \Bigr ].
\end{align}

\begin{lemma}\label{lem:renormalised_eq}
	Let $Z \in \mathcal{M}^s$ be a smooth model over the smooth noise $\xi$. Then we have that $u = \mathcal{R}^Z \mcV + u^\psi$ solves
	\begin{align}
		(\partial_t - \Delta) u = \Gamma(u) (\partial_x u)^2 + g(u) \partial_x u + h(u) + \sigma(u) \xi + a(y) \mathfrak{C}(y,w,\partial_x y, \partial_x w)
	\end{align}
	where $w = \mathcal{R}^Z \mcW$ and $\mathfrak{C}$ is the counterterm associated to the system \eqref{eq:W_sys} via the model $Z$.
\end{lemma}
\begin{proof}
Since $a' = - a^2 \Gamma \circ Q$ we have that
$$a D\mcW + a' \mcW D\mcY = a [D\mcW - a  \mcW D \mcY \Gamma \circ Q].$$ 
We also note that $Q' = a$ and therefore
\begin{align*}
	\mcW &= \mcP_\gamma \mathbf{1}_+ b \Bigl [ a^2 \Gamma(a\mcW + Q) [ D \mcW - a \mcW D\mcY \Gamma \circ Q]^2 + 2 \tilde{\Gamma}^0 (a \mcW, Q)
	\\
	& \qquad \times a^3 \mcW  D \mcY [ D \mcW - a  \mcW D \mcY \Gamma \circ Q] + \tilde{\Gamma}^1(a \mcW, Q) a^4 \mcW^2 
	\\
	& \qquad \times (D \mcY)^2 + a g(a \mcW + Q) [ D\mcW - a \mcW D \mcY \Gamma \circ Q 
	\\ 
	& \qquad + D \mcY] + h(a \mcW + Q) + \sigma(a \mcW + Q) \Xi \Bigr ].
\end{align*}
We now note that $a(\mcY) b(\mcY) = 1$. Indeed, this relation is obvious from the definition of $a$ and $b$ at the level of the reconstructions and thus follows for the modelled distributions from the injectivity of the reconstruction operator on the polynomial sector. Therefore, written in the form required to apply the results of \cite{BCCH}, the non-linearities for $\mcW$ are
\begin{align*}
F_W^0(\mcY, \mcW, D\mcY, D\mcW) &= \Bigl [ a \Gamma(a\mcW + Q) [D \mcW - a \mcW D\mcY \Gamma \circ Q]^2 + 2 \tilde{\Gamma}^0 (a\mcW, Q)
	\\
	& \qquad \times a^2 \mcW  D \mcY [ D \mcW - a \mcW D \mcY \Gamma \circ Q] + \tilde{\Gamma}^1(a \mcW, Q) a^3 \mcW^2 
	\\
	& \qquad \times (D \mcY)^2  + g(a \mcW + Q) [ D\mcW - a  \mcW D \mcY \Gamma \circ Q
	\\ 
	& \qquad + D \mcY] + b h(a \mcW + Q) \Bigr ]
	\\
	F_W^\Xi(\mcY, \mcW, D\mcY, D\mcW) &= b \sigma(a \mcW + Q)
\end{align*}
We now let $u = a(y) w + Q(y)$ where $w = \mathcal{R} \mcW$ and $y = \mcR \mcY$.

We compute
\begin{align*}
	\partial_x u &= a(y) \partial_x w - a(y)^2 w \partial_x y \Gamma(Q(y)) + Q'(y) \partial_x y
	\\
	&= a(y) \partial_x w - a(y)^2 w \partial_x y \Gamma(Q(y)) + a(y) \partial_x y.
\end{align*}
The main result of \cite{BCCH} then implies that $w$ solves
\begin{align*}
	(\partial_t - \Delta) w &= b(y) \Gamma(u) (\partial_x u - a(y) \partial_x y)^2 + 2 [\Gamma(u) - \Gamma(Q(y))] \partial_x y (\partial_x u - a(y) \partial_x y) 
	\\
	& \qquad + a [ \Gamma(u) - \Gamma(Q(y)) - \Gamma'(Q(y)) a(y) w] (\partial_x y)^2 + b(y) g(u) \partial_x u  + b(y) h(u) 
	\\ 
	& \qquad + b(y) \sigma(u) \xi + \mathfrak{C}(y,w, \partial_x y, \partial_x w)
	\\
	&= b(y) \Bigl [ \Gamma(u) (\partial_x u - a(y) \partial_x y)^2 + 2 [\Gamma(u) - \Gamma(Q(y))] a \partial_x y (\partial_x u - a(y) \partial_x y) 
	\\
	& \qquad +  [ \Gamma(u) - \Gamma(Q(y)) - \Gamma'(Q(y)) a(y) w] a^2 (\partial_x y)^2 + g(u) \partial_x u  + h(u) 
	\\ 
	& \qquad +  \sigma(u) \xi + b(y)^{-1}  \mathfrak{C}(y,w, \partial_x y, \partial_x w) \Bigr ]
\end{align*}
where $\mathfrak{C}(y,w, \partial_x y, \partial_x w)$ is the counterterm. For now, the precise form of the counterterm is not important so we just carry the placeholder through the calculation. 

Now if $w = bv$ we have that
\begin{align*}
	(\partial_t - \Delta) w = b(\partial_t - \Delta) v + v (\partial_t - \Delta) b - 2 \partial_x b \partial_x v
\end{align*}
where, by construction, we have
\begin{align*}
	(\partial_t - \Delta) b(y) &= - b(y) \Gamma'(u^\psi) (\partial_x u^\psi)^2 
	\\
	\partial_x b(y) &= b(y) \Gamma(u^\psi) \partial_x u^\psi
\end{align*}
where we have used the fact that $Q(y) = u^\psi$. 

Therefore 
\begin{align*}
	(\partial_t - \Delta) v = b(y)^{-1} (\partial_t - \Delta) w + \Gamma'(u^\psi) v (\partial_x u^\psi)^2 + 2 \Gamma(u^\psi) (\partial_x u^\psi) (\partial_x v).
\end{align*}

We now note that $\partial_x u^\psi = \partial_x Q(y) = a(y) \partial_x y$. Therefore
\begin{align*}
	(\partial_t - \Delta) v &= \Gamma(u) (\partial_x u - \partial_x u^\psi)^2 + 2 [\Gamma(u) - \Gamma(u^\psi)] \partial_x u^\psi (\partial_x u - \partial_x u^\psi) \\
	& \qquad + [\Gamma(u) - \Gamma(u^\psi) - \Gamma'(u^\psi) v ](\partial_x u^\psi)^2 + g(u) \partial_x u + h(u) 
	\\
	& \qquad + \sigma(u) \xi + b(y)^{-1} \mathfrak{C}(y, w, \partial_x y, \partial_x w) + \Gamma'(u^\psi) v (\partial_x u^\psi)^2 + 2 \Gamma(u^\psi) (\partial_x u^\psi)(\partial_x v) 
	\\
	&= \Gamma(u) (\partial_x u)^2 + g(u) \partial_x u + h(u) + \sigma(u) \xi + a(y) \mathfrak{C}(y, w, \partial_x y, \partial_x w) 
	\\
	& \qquad - \Gamma(
	u^\psi) (\partial_x u^\psi)^2.
\end{align*}
Since $u = v + u^\psi$ where $u^\psi$ solves
\begin{align*}
	(\partial_t - \Delta) u^\psi = \Gamma(u^\psi)(\partial_x u^\psi)^2
\end{align*}
this completes the proof. 
\end{proof}
Therefore to prove Lemma~\ref{lem:gKPZ_identification}, it remains to show that the counterterm for gKPZ driven by the smooth symmetric model $Z$ is the same as $a(y) \mathfrak{C}(y, w, \partial_x y , \partial_x w)$.

We begin by recalling that by \cite{BCCH}, the counterterms $\mathfrak{C}$ are of the form
\begin{align*}
	\sum_{|\tau|_\fs < 0} \frac{\Upsilon_\diamond[\tau]}{S(\tau)} \ell(\tau)
\end{align*}
for $\diamond \in \{u, w\}$.
\begin{lemma}\label{lem:symm_reduction}
	We have that $\Upsilon_w[\tau] = 0$ unless $\tau$ contains no edges corresponding to $\mcY$, $D\mcY$. Furthermore, for $\diamond \in \{u, w\}$, we have that $\frac{\Upsilon_\diamond[\tau]}{S(\tau)} \ell(\tau) \neq 0$ implies that every vertex in $\tau$ has outgoing edges appearing in one of the two forms
	\begin{align*}
		\<d2j> \qquad \<jn>
	\end{align*}
\end{lemma}
\begin{remark}
In the case of the equation for $u$ this is a well-known fact about the consequences of symmetry for the gKPZ equation. The novelty in the previous statement is that the same restriction holds for the counterterm for $\mcW$.
\end{remark}
\begin{proof}
We first note that in the case $\diamond = w$, no tree with an edge corresponding to the $\mcY$ equation contributes since the corresponding non-linearity is $0$.

The next step is to remove the majority of contributions by appealing to the fact that $\ell(\tau) = 0$ whenever $e + k_x$ is odd where $e$ denotes the number of edges in $\tau$ carrying a derivative and $k_x$ is the total degree of spatial polynomials of $\tau$. 

We write $n$ for the number of noises in $\tau$ and let $m$ count the number of internal vertices of $\tau$ that do not carry a noise. Since $\tau$ is a tree with $n+m$ vertices, there are $n + m - 1$ kernel edges. We thus see that
$$|\tau|_\fs = n \alpha + 2 (n + m - 1) - e + |k|_\fs = - 2 + n (\alpha + 2) + (2m - e) + |k|_\fs.$$
Since $\Upsilon_\diamond[\tau] \neq 0$ implies that $e \le 2m$ (since only an internal vertex with no noise may be followed by a derivative edge and then at most two due to the form of the non-linearity), each term in brackets is non-negative. Since $\alpha + 2 > 0$, we see that $2m - e + |k|_\fs \in \{0,1\}$. In particular, the polynomial decoration consists of at most one instance of $X_1$ and no polynomials in the time direction. In the case where $k_x = 1$, we have that $2m - e = 0$ so that $e$ is even. Thus $e + k_x$ is odd and $\ell(\tau) = 0$. On the other hand, if $k_x = 0$ and $2m - e \neq 0$ then we similarly have that $e + k_x = e$ is odd since in that case $e = 2m - 1$. This completes the proof.
\end{proof}

\begin{lemma}\label{lem:Upsilon}
	For every $\tau$ satisfying the constraints of the previous lemma, we have that
	\begin{align}\label{eq:counterterm_claim}
	\Upsilon_u[\tau] = b^{-1} \Upsilon_w[\tau].
	\end{align}
\end{lemma}
\begin{proof}
By the previous lemma, we may assume that every vertex in $\tau$ has one of the two forms $\<d2j>, \<jn>$. We proceed by induction in the number of edges.
We write $F_W^0(y,w, y',w'), F_W^\Xi(y,w,y',w')$ for the non-linearities for the $W$ equation populated by dummy variables. We similarly write $F_U^0[u, u'] = \Gamma(u) (u')^2 + g(u) (u') + h(u)$ and $F_U^\Xi[u, u'] = \sigma(u)$ for the non-linearities for gKPZ. Since we have reduced to trees that contain only edges corresponding to $w, w'$ (or $u, u'$ in the case of the equation for $u$), it suffices to compute
\begin{align}\label{eq:derivative_computation}
	\partial_{w'}^j \partial_w^k F_W^\Xi[y,w,y', w'] &= \delta_{j, 0} b^{1-k} \sigma^{(k)} (u) = b^{1-k} \partial_{u'}^j \partial_u^k F_U^\Xi[u, u']
	\\ \nonumber
	\partial_w^j \partial_{w'}^2 F_W^0[y,w,y',w'] &= b^{-1-j} 2\Gamma^{(j)}(u) = b^{-1-j} \partial_u^j \partial_{u'}^2 F_U^0[u, u'].
\end{align}
We then proceed to prove \eqref{eq:counterterm_claim} by induction in the number of internal vertices in $\tau$. 

The base case is $\tau = \Xi$, in which case the result is true from the definitions of $F_W^\Xi$ and $F_U^\Xi$. Note that we do not need to consider the case $\tau = \mathbf{1}$ since no tree under consideration has $\mathbf{1}$ as a leaf. For the induction step, we consider separately the case where the root of $\tau$ has type $\<d2j>$ and the case where it has type $\<jn>$ which correspond to $\tau$ taking one of the two forms
\begin{align*}
	\tau =  \mcI^\prime \tau_{-1} \mcI^\prime \tau_0 \prod_{i = 1}^j \mcI \tau_i, \qquad \tau = \Xi \prod_{i=1}^j \mcI \tau_i.
\end{align*}
In the latter case 
\begin{align*}
	b(y)^{-1} \Upsilon_w[\tau](y,w,y',w') &= b(y)^{-1} \prod_{i=1}^j \Upsilon_w[\tau_i](y,w,y',w') \partial_w^j F_W^\Xi[y,w,y',w'] 
	\\
	&= b(y)^{-j} \sigma^{(j)}(u) \prod_{i=1}^j b(y) \Upsilon_u[\tau_i](u,u') 
	\\
	&= \Upsilon_u[\tau](u,u')
\end{align*}
where the first line follows from the definition of $\Upsilon_w$, the second line follows from the induction hypothesis and \eqref{eq:derivative_computation} and where the last line follows from the definition of $\Upsilon_u$.

For the remaining case where $\tau =  \mcI^\prime \tau_{-1} \mcI^\prime \tau_0 \prod_{i = 1}^j \mcI \tau_i$, we similarly compute
\begin{align*}
	b(y)^{-1} & \Upsilon_w[\tau](y,w,y',w') 
	\\ &= b(y)^{-1} \Upsilon_w[\tau_{-1}]\Upsilon_w[\tau_0] \prod_{i=1}^j \Upsilon_w[\tau_i](y,w,y',w') \partial_w^j \partial_{w'}^2 F_W^0[y,w,y',w'] 
	\\
	&= b(y)^{-1} b(y)^{2 + j} \Bigl ( \prod_{i=-1}^j \Upsilon_u[\tau_i](u,u') \Bigr ) b(y)^{-1-j} 2\Gamma^{(j)}(u) 
	\\
	&= \Upsilon_u[\tau](u,u')
\end{align*}
where we used a similar chain of reasoning as in the previous case. This completes the proof.
\end{proof}

\bibliographystyle{Martin.bst}
\bibliography{BPHZ}
\end{document}